\documentclass[notitlepage,11pt]{article}%
\usepackage{amsfonts,amsmath,amssymb}
\usepackage{graphicx}
\usepackage{multicol}
\usepackage{breakcites}%
\usepackage{natbib}

\usepackage{enumitem}

\usepackage{mathtools}
\DeclarePairedDelimiter\abs{\lvert}{\rvert}

\usepackage{longtable, booktabs}
\usepackage{ltcaption}

\usepackage{scalefnt}
\usepackage{pdflscape}
\usepackage[multiple]{footmisc}

\usepackage{float}  

\providecommand{\U}[1]{\protect\rule{.1in}{.1in}}
\newtheorem{theorem}{Theorem}

\newtheorem{definition}[theorem]{Definition}

\newtheorem{problem}[theorem]{Problem}
\newtheorem{proposition}[theorem]{Proposition}
\newtheorem{remark}[theorem]{Remark}

\newtheorem{assumption}[theorem]{Assumption}
\newenvironment{proof}[1][Proof]{\noindent\textbf{#1.} }{\ \rule{0.5em}{0.5em}}

\begin{document}

\title{Generalized exponential basis for efficient solving of homogeneous diffusion free boundary problems: Russian option pricing}
\author{Igor V. Kravchenko\\{\footnotesize {Instituto Universit\'{a}rio de Lisboa (ISCTE-IUL),
Edif\'{\i}cio II, Av. Prof. An\'{\i}bal Bettencourt,}}\\{\footnotesize {1600-189 Lisboa, Portugal}} \\
\footnotesize{Business Research Unit (BRU-IUL), Lisboa, Portugal}\\
\and Vladislav V. Kravchenko\\
\footnotesize{\footnotesize {Departamento de Matem\'{a}ticas, CINVESTAV del IPN, Unidad
Quer\'{e}taro, Libramiento}}\\{\footnotesize {Norponiente No. 2000, Fracc. Real de Juriquilla,
Quer\'{e}taro, Qro. C.P. 76230 M\'{e}xico}}\\
\and Sergii M. Torba\\{\footnotesize {Departamento de Matem\'{a}ticas, CINVESTAV del IPN, Unidad
Quer\'{e}taro, Libramiento}}\\{\footnotesize {Norponiente No. 2000, Fracc. Real de Juriquilla,
Quer\'{e}taro, Qro. C.P. 76230 M\'{e}xico}}
\and Jos\'{e} Carlos Dias\\{\footnotesize {Instituto Universit\'{a}rio de Lisboa (ISCTE-IUL),
Edif\'{\i}cio II, Av. Prof. An\'{\i}bal Bettencourt,}}\\{\footnotesize {1600-189 Lisboa, Portugal}} \\\footnotesize{Business Research Unit (BRU-IUL), Lisboa, Portugal}
}

\date{}

\maketitle

\begin{abstract}
This paper develops a method for solving free boundary problems for time-homogeneous diffusions. We combine the complete exponential system of solutions for the heat equation, transmutation operators and recently discovered Neumann series of Bessel functions representation for solutions of Sturm-Liouville equations to construct a complete system of solutions for the considered partial differential equations. The conceptual algorithm for the application of the method is presented. The valuation of Russian options with finite horizon is used as a numerical illustration. The solution under different horizons is computed and compared to the results that appear in the literature.
\end{abstract}


\section{Introduction}

One of the approaches for solving boundary value problems for partial
differential equations (PDE's) is based on complete systems of solutions
(CSS). In particular, several CSS have been used in different models such as:
fundamental solutions (the well known method of fundamental solutions or
discrete sources) in \cite{Kupradze1967}, \cite{Alexidze1991}, \cite{Fairweather1998} and \cite{Doicu2000}; heat polynomials in \cite{colton1976solution}, \cite{Reemtsen1982},
\cite{colton1984numerical}, \cite{sarsengeldin2014analytical}  and
\cite{KravchenkoOtero2017}; wave polynomials in \cite{khmelnytskaya2013wave}
among many others. For the present paper the following family $\{e_{n}^{\pm
}\}_{n\in\mathbb{N}}$ of exponential solutions of the heat equation
\begin{equation}
h_{xx}=h_{t},\label{eq_heatEq}%
\end{equation}
defined as
\begin{equation}
e_{n}^{\pm}\left(  x,t\right)  =\exp(\pm i\omega_{n}x-\omega_{n}%
^{2}t),\label{eq_en_trig}%
\end{equation}
are of particular interest. Here the constants $\omega_{n}$ are chosen such
that the limit
\begin{equation}
d:=\lim_{n\rightarrow\infty}\frac{n}{\omega_{n}^{2}}>0\label{eq_cond_w_conv}%
\end{equation}
exists. In \cite{colton1980analytic}, the completeness of this system of
solutions was proved for bounded domains satisfying certain smoothness
properties.

As a rule, the approach based on CSS cannot be directly applied to equations
with variable coefficients, because CSS are not available in a closed form. In
\cite{colton1976solution}, it was developed the idea to extend the approach
of CSS to equations with variable coefficients with the aid of transmutation
operators whenever they are known or can be constructed efficiently. However,
the construction of the transmutation operators is itself a complicated task.

In the present paper, we propose the construction of the CSS generalizing
exponential solutions \eqref{eq_en_trig} for the equation
\begin{equation}
\mathbf{C}u(y,t):=\frac{1}{w\left(  y\right)  }\left(  \frac{\partial
}{\partial y}\left(  p\left(  y\right)  \frac{\partial}{\partial y}\right)
-q\left(  y\right)  \right)  u(y,t)=u_{t}(y,t).\label{eq_parabolIntro}%
\end{equation}

These generalized exponential solutions represent a CSS for equation
\eqref{eq_parabolIntro} and are the images of the exponential solutions
\eqref{eq_en_trig} under the action of the transmutation operator. Moreover,
they can be computed by a simple robust recursive integration procedure which
does not require the knowledge of the transmutation operator itself. This
makes possible to extend the numerical methods (minimization problems) for
free boundary problems (FBP's) for the heat equation to the time-homogeneous
parabolic equations, in particular, to the finite horizon Russian option
(FHRO) valuation problem that we analyze in detail in this paper.

In \cite{KrKrTorba2017}, a numerical method was developed for the classical one dimensional Stefan like problem for the time-homogeneous parabolic operator using the CSS of the transmuted heat polynomials, that was referred to as THP method. It is well known that the CSS based on  polynomials result in badly conditioned matrices, making the application of THP complicated for the practical computations. This is the case for the FHRO. Fortunately, there are alternative CSS for the heat equation \eqref{eq_heatEq}, for which we also know their transmuted images.

In practice, the FBP's are often challenging for numerical methods. For example, the boundary conditions arising in relation to the FHRO problem are
non consistent (the solution or its derivative can not be continuous along the boundary). This leads to all sort of different computational issues. We
present a step by step algorithm and discuss the numerical issues that we have encountered. The method that we propose takes into account known properties of the solution (such as monotone increase of the free boundary) and of the functions from the CSS (possibility to automatically satisfy one of the boundary conditions) making the computations easier and more predictable.


Even though there are several quantitative studies on the FHRO, e.g.
\cite{duistermaat2005finite}, \cite{kimura2008valuing} and
\cite{jeon2016integral}, it seems that there is no agreement on the exact
value for the option. We contribute to this discussion confirming the values from \cite{jeon2016integral} and providing possible explanation of the discrepancy with \cite{kimura2008valuing}.

\begin{figure}
\centering
  \includegraphics*[trim= 0 0 0 0,scale=1.00]{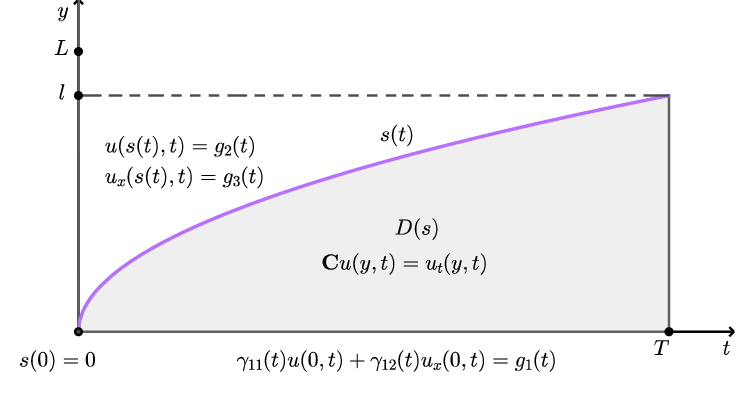}
  \caption{Free boundary problem.}
  \label{fig_FBP}
\end{figure}


The parabolic FBP's arise in many fields, and hence the method proposed has a lot of potential for further applications and developments. In particular, for the financial engineering applications presenting path-dependence and early exercise features such as lookback options, American options, etc. In this paper, for the FHRO, we are restricted to the \cite{BlackScholes1973} and \cite{Merton1973} (BSM) model (and respective infinitesimal generator) since it is not clear how to generalize the problem to different diffusions and keep the resulting FBP two dimensional (see also \cite{kamenov2008bachelier} for Bachelier model). However, for other financial (and non-financial) applications, where the FBP can be formulated using a general operator \eqref{eq_parabolIntro}, our method can be applied as well. This is, for example, the case of American option where the underlying asset follows a time-homogeneous diffusion process.

The paper is structured as follows. In Section \ref{sec_FBP}, we state the FBP. In Section \ref{sec_TransmOp}, we introduce the transmutation operators and highlight some of the relevant theoretical results. In Section \ref{sec_TransCSS}, we introduce the notion of the CSS and see how it
can be used to approximate the solutions of the PDE \eqref{eq_parabolIntro}. We also show how to explicitly construct the transmuted CSS for the case of the generalized trigonometric series. In Section \ref{sec_MinProblem}, we state the minimization problem and summarize an algorithm for the solution. In Section \ref{sec_RO}, we introduce the FHRO and set-up the corresponding FBP. The quantitative results for the FHRO, the discussion of the numerical issues and the comparison with existing in the literature results are presented in Section \ref{sec_NumResults}. Section \ref{sec_Conslusion} presents some concluding remarks.

\section{\label{sec_FBP}The free boundary problem}

Consider the differential expression $\mathbf{C}$ from \eqref{eq_parabolIntro}
where the functions $p$, $q$ and $w$ satisfy the following assumption.

\begin{assumption}\label{Assum_pqw} The functions $p$, $p^{\prime}$, $q$, $w$ and $w^{\prime}$ are real valued and continuous on $[0,L]$. Additionally, it is assumed that $p^{\prime}$ and $w^{\prime}$ are absolutely continuous and that $p>0$ and $w>0$.
\end{assumption}

Every non-negative function $s\in C^{1}\left[ 0,T\right] $, such that $s\left( 0\right) =0$ and $0<s(t)\le L$, $t\in (0,T]$, defines a domain
\begin{equation}
D(s) =\{(y,t)\in\mathbb{R}^{2}:0<y<s(t),0<t<T\} ,
\label{D(s)_1}
\end{equation}
as shown on Figure \ref{fig_FBP}.

\begin{problem}
\label{Problem_FBP_1}Find functions $u(y,t)$ and $s(t)$ such that%
\begin{align}
\mathbf{C}u(y,t) &=u_{t}(y,t) ,\qquad (y,t) \in D(s) ,  \label{eq_FBP1_1}\\
\gamma _{11}(t)u(0,t) +\gamma _{12}(t)u_{y}(0,t) &=g_{1}(t) ,\qquad t\in (0,T) ,  \label{eq_FBP_b2_1} \\
u(s(t),t) &=g_{2}(t) ,\qquad t\in (0,T) ,  \label{eq_FBP_b3_1} \\
u_{y}(s(t),t) &=g_{3}(t) ,\qquad t\in (0,T) ,  \label{eq_FBP_b4_1}
\end{align}%
where $\gamma _{1j}$ for $j\in\{1,2\} $ and $g_{k}$ for $k\in\{1,2,3\}$ are analytic functions.
\end{problem}

The aim of this paper is to illustrate the application of the numerical method based on the transmutation operators theory to Problem \ref{Problem_FBP_1}. To avoid the questions of the existence and uniqueness of solution specific to each problem, we will make the following assumption.

\begin{assumption}\label{Assump_UniqueSol}
In suitable classes of functions, there exists a unique solution to Problem \ref{Problem_FBP_1}.
\end{assumption}

The basic idea of a numerical method based on a CSS is that any linear combination of the functions from the CSS is already a solution to \eqref{eq_FBP1_1}. Hence one may construct the linear combination that will satisfy (approximately) the boundary conditions of Problem \ref{Problem_FBP_1}. As was mentioned in the introduction, for many practical problems the boundary conditions are inconsistent resulting that the uniform norm is not a choice for measuring the quality of an approximate solution, and some kind of $L_2$ norm is more convenient. For this reason we will make the following assumption guaranteeing the proposed numerical method to work.

\begin{assumption}\label{Continuous depSol}
The solution to Problem \ref{Problem_FBP_1} continuously depends on the boundary data in a suitable $L_2$ norm.
\end{assumption}

\begin{remark}
This problem includes as a special case the classical degenerate one dimensional Stefan problem. For these types of problems the dependence of the functions $g_{2}$ and $g_{3}$ on the function $s$ and its derivatives can be specified---see \cite{rose1960method} for example. For our method this does not represent additional difficulty. The definition of Problem \ref{Problem_FBP_1} may also include additional conditions that can be necessary to guarantee the existence and the uniqueness of solution. We will see this in the example for the FHRO constructed further.
\end{remark}

\section{\label{sec_TransmOp}Transmutation operators}

In this section we present our main operational tool: the transmutation
operator.

\begin{definition}\label{Def Transmutation}
Let $E_{1}$ and $F_{1}$ be linear subspaces of the linear topological spaces $E$ and $F$, respectively. Consider the pair of operators $\mathbf{A}:E_{1}\rightarrow E$ and $\mathbf{B}:F_{1}\rightarrow F$. A linear invertible operator $\mathbf{T}:F\to E$ defined on the whole $F$ is called a \textbf{transmutation operator} for the pair of operators $\mathbf{A}$ and $\mathbf{B}$ if the following conditions are met:

\begin{enumerate}
\item The operator $\mathbf{T}$ is continuous in $F$, its inverse $\mathbf{T}^{-1}$ is continuous in $E$;

\item $T(F_{1})\subset E_{1}$;

\item The following operator equality is valid%
\begin{equation*}
\mathbf{AT}=\mathbf{TB},
\end{equation*}
or which is the same
\begin{equation*}
\mathbf{A}=\mathbf{TBT}^{-1}.
\end{equation*}
\end{enumerate}
\end{definition}

We are particularly interested in the case of $\mathbf{A}$ being the differential operator $\mathbf{C}$ defined in \eqref{eq_parabolIntro} and $%
\mathbf{B}$ being the second derivative. The idea is to transmute the solutions of the heat equation \eqref{eq_heatEq} into the solutions of the parabolic equation \eqref{eq_FBP1_1}.\footnote{As an illustration, let $h(x,t)$ be a solution of \eqref{eq_heatEq}, then if the operator $\mathbf{T}$ exists, $u=\mathbf{T}h$ will be the solution to equation \eqref{eq_FBP1_1}, indeed $\mathbf{C}u-u_{t}=\mathbf{CT} h-\partial _{t}\mathbf{T}h=\mathbf{T(}\partial _{yy}h-\partial _{t}h)=0$.} Throughout this section we consider equation \eqref{eq_parabolIntro} to be defined for $y\in[A,B]$, and Assumption \ref{Assum_pqw} to hold on the segment $[A,B]$.

In the work of \cite{kravchenko2016liouville} and  \cite{KravchenkoTorba_Neumann_2018}
using the Liouville transformation
\begin{equation*}
x=l(y):=\int_{A_0}^{y}(w(s)/p(s))^{1/2}ds, \qquad y\in \lbrack A,B],
\end{equation*}
where the point $A_0$ is chosen such that
\[
\int_{A}^{A_0}(w(s)/p(s))^{1/2}ds = \int_{A_0}^{B}(w(s)/p(s))^{1/2}ds =: b,
\]
the transmutation operator for the operators $\mathbf{C}$ and $%
\partial _{xx}$ was studied, for the spaces $E_{1}=C^{2}[A,B]$, $E=C[A,B]$, $
F_{1}=C^{2}[-b,b]$ and $F=C[-b,b]$.

\begin{remark}
Equation \eqref{eq_FBP1_1} is a separable PDE, which implies that we only have to construct a one-dimensional transmutation operator for the operator $ \mathbf{C}$.
\end{remark}

The transmutation operator $\mathbf{T}$ is known in closed form only for few equations \eqref{eq_FBP1_1}. However, as we will show for the construction of the CSS, the knowledge of the operator $\mathbf{T}$ itself is not indispensable. This construction is based on the fundamental result Theorem \ref{Teor_Phi_x} that connects the images of the transmutation operator to the family of the recursive integrals, that are called \textbf{formal powers}, see Definition \ref{Def_Phik} below.

Let us define an auxiliary function
\begin{equation*}
\rho \left( y\right) =\left[ p\left( y\right) w\left( y\right) \right]^{1/4},
\end{equation*}
and let $f$ be a non-vanishing (in general, complex-valued) solution of the equation
\begin{equation}
\left( p\left( y\right) f^{\prime }\left( y\right) \right) ^{\prime
}-q\left( y\right) f\left( y\right) =0,\qquad y\in \left[A,B\right] ,
\label{eq_PQR_1}
\end{equation}%
with an initial condition set as
\begin{equation}
f\left( A_0\right) =\frac{1}{\rho \left( A_0\right) }.  \label{eq_PQR_2}
\end{equation}
Since $p$ and $q$ satisfy Assumption \ref{Assum_pqw}, equation \eqref{eq_PQR_1} has two linearly independent regular solutions $f_{1}$ and $f_{2}$ whose zeros alternate. We may construct a non-vanishing solution as $f=f_{1}+if_{2}$ ---\citet[Remark 5]{KravchenkoPorter10}
\begin{definition}
\label{Def_Phik}Let $p$, $q$, $w$ satisfy Assumption \ref{Assum_pqw} and let $f$ be a non-vanishing solution of equation \eqref{eq_PQR_1} that satisfies
condition \eqref{eq_PQR_2}. Then, the associated formal powers are defined, for $k=0,1,2,...$, as
\begin{equation*}
\Phi_{k}(y) =
\begin{cases}
f\left( y\right) Y^{\left( k\right) }\left( y\right) , & k\text{ odd} \\
f\left( y\right) \tilde{Y}^{\left( k\right) }\left( y\right) , & k\text{
even}
\end{cases},
 \qquad \Psi_{k}\left( y\right) =
\begin{cases}
\frac{1}{f\left( y\right) }Y^{\left( k\right) }\left( y\right) , & k\text{
even} \\
\frac{1}{f\left( y\right) }\tilde{Y}^{\left( k\right) }\left( y\right) ,
& k\text{ odd}%
\end{cases}
,
\end{equation*}
where two families of the auxiliary functions are defined as%
\begin{align*}
Y^{\left( 0\right) }\left( y\right) & \equiv\tilde{Y}^{\left( 0\right)
}\left( y\right) \equiv1, \\
\displaybreak[2]
Y^{\left( k\right) }\left( y\right) & =
\begin{cases}
k\int_{A_0}^{y}Y^{\left( k-1\right) }\left( s\right) \frac{1}{f^{2}\left(
s\right) p\left( s\right) }ds,   & k\text{ odd} \\
k\int_{A_0}^{y}Y^{\left( k-1\right) }\left( s\right) f^{2}\left( s\right)
p\left( s\right) ds, &   k\text{ even}%
\end{cases}
, \\
\displaybreak[2]
\tilde{Y}^{\left( k\right) }\left( y\right) & =
\begin{cases}
k\int_{A_0}^{y}\tilde{Y}^{\left( k-1\right) }\left( s\right) f^{2}\left(
s\right) p\left( s\right) ds,  & k\text{ odd} \\
\displaybreak[2]
k\int_{A_0}^{y}\tilde{Y}^{\left( k-1\right) }\left( s\right) \frac{1}{%
f^{2}\left( s\right) p\left( s\right) }ds, & k\text{ even}%
\end{cases}
.
\end{align*}
\end{definition}

\begin{theorem}[\cite{kravchenko2016liouville}]
\label{Teor_Phi_x} Let $p$, $q$ and $w$
satisfy Assumption \ref{Assum_pqw} for all $y\in \lbrack A,B]$ and let $%
f$ be a non-vanishing solution of equation \eqref{eq_PQR_1} that satisfies
condition \eqref{eq_PQR_2}, then there exists a unique complex valued function $K$ and the transmutation operator $\mathbf{T}$ defined as
\begin{equation}
\mathbf{T}h(y)=\frac{h(l(y))}{\rho (y)}+\int_{-l(y)}^{l(y)}K(y,t)h(t)dt, \label{TransmutProperty}
\end{equation}
for $h\in C[-b,b]$, and satisfying the equality
\begin{equation*}
\mathbf{CT}h=\mathbf{T}\partial _{xx}h,
\end{equation*}
for any $h\in C^{2}[-b,b]$ such that
\begin{equation*}
\mathbf{T}[1]=f(y).
\end{equation*}
Moreover, for any $n\in N\cup \left\{ 0\right\} $%
\begin{equation}
\mathbf{T}\left[ x^{n}\right] =\Phi _{n}\left( y\right)  \label{eq_Txn}
\end{equation}
and for $u=\mathbf{T}h$ the following boundary conditions are satisfied
\begin{align}
u(A_{0}) &=\frac{h(0)}{\rho(A_{0})}   \label{eq_Txn_BC1} \\
u'(A_{0}) &=h(0)f'(A_{0})+h'(0)\dfrac{1}{\rho(A_{0})}\sqrt{\dfrac{w(A_{0})}{p(A_{0})}}. \label{eq_Txn_BC2}
\end{align}

\end{theorem}

The theorem provides tools for computation of the transmuted powers. It was used directly in \cite{KrKrTorba2017} for the application
of the Transmuted heat polynomials (THP)\ method to the Stefan-like problem. In this paper, we will use a different CSS.

\begin{remark}\label{Rem_HalfInterval}
This transmutation operator $\mathbf{T}$ has the following important property. Consider a function $u=\mathbf{T}v$. Then the values $u(y)$ for $y\in [A_0,B]$ are completely determined by the function $v$ and the values of $p$, $q$, $w$ on the segment $[A_0,B]$ and are independent of the values of $p$, $q$, $w$ on $[A,A_0)$. For this reason we may consider the restriction of equation \eqref{eq_FBP1_1} onto $[A_0,B]$ and the operator $\mathbf{T}$ as the operator mapping functions from $C[-b,b]$ to functions from $C[A_0,B]$. Such operator is no longer invertible, however it is continuous and maps a solution of the heat equation into a solution of \eqref{eq_FBP1_1} and is sufficient to present the proposed numerical method. Moreover, it allows one to take into account the boundary conditions \eqref{eq_Txn_BC1} and \eqref{eq_Txn_BC2}. For that reason from now on we assume that $A_0=A$ in the Liouville transformation, and when we need the invertibility of $\mathbf{T}$, we continue the coefficients $p$, $q$, $w$ to the left arbitrarily asking only that Assumption \ref{Assum_pqw} be fulfilled. Moreover, in the rest of the present paper we consider $A_0=0$.
\end{remark}

\section{\label{sec_TransCSS}Transmutation of the complete systems of
solutions}
Let $D=\{(y,t):y_{1}(t)<y<y_{2}(t),t\in (0,T]\}$, where $0\le y_{i}(t)\le L$,  $i\in\{1,2\}$, are continuous functions, be a subset of $\mathbb{R}^{2}$.
\begin{definition}\label{Def_CSS_global}
The set of solutions $\{u_{n}\}_{n\in \mathbb{N}}$ of equation \eqref{eq_FBP1_1}
is said to be a \textbf{complete system of solutions in the closed region $\bar D$} if for any $u\in
C(\bar{D})\cap C^{2,1}(D)$, a solution to \eqref{eq_FBP1_1}, and for
any $\varepsilon >0$ there exist an integer $N=N(\varepsilon )$ and constants $%
a_{0},...,a_{N}$ such that
\begin{equation*}
\max_{(y,t)\in \bar{D}}\left\vert
u(y,t)-u^{N}(y,t)\right\vert <\varepsilon,
\end{equation*}
where
\begin{equation}\label{eq_uN_H}
u^{N}(y,t) = \sum\limits_{n=0}^{N}a_{n}u_{n}(y,t).
\end{equation}
\end{definition}

The completeness of a system of functions in the sense of Definition \ref{Def_CSS_global} may be difficult to establish, and the following weaker form of the definition may be sufficient for practical applications.

\begin{definition}\label{Def_CSS_compacts}
The set of solutions $\{u_{n}\}_{n\in \mathbb{N}}$ of equation \eqref{eq_FBP1_1} is said to be a \textbf{complete system of solutions} if for any $u\in
C^{2,1}(D)$, a solution to \eqref{eq_FBP1_1}, for any compact subset $K\subset D$ and for any $\varepsilon >0$ there exist an integer $N = N(\varepsilon, K)$ and constants $a_{0},...,a_{N}$ such that
\begin{equation*}
\max_{(y,t)\in K}\left\vert
u(y,t)-u^{N}(y,t)\right\vert <\varepsilon.
\end{equation*}

\end{definition}

The following proposition allows us, on the basis of the CSS for the heat equation, to construct the CSS for equation \eqref{eq_FBP1_1}. We define
\[
b= \int_{0}^{L}(w(s)/p(s))^{1/2}ds.
\]

\begin{proposition}
\label{Prop Complete system_1} Let $\{v_{n}\}_{n\in\mathbb{N}}$ be a CSS for the heat equation on a rectangle $[-b,b]\times [\delta,T]$ for all sufficiently small $\delta>0$. Consider the system of the transmuted functions $\{u_{n}\}_{n\in \mathbb{N}}$, i.e.
\begin{equation}\label{Transmuted CSS}
u_{n}=\mathbf{T}[v_{n}],
\end{equation}%
where $\mathbf{T}$ is defined in Theorem \ref{Teor_Phi_x} (see Remark \ref{Rem_HalfInterval}).
Then the system $\{u_{n}\}_{n\in \mathbb{N}}$ is a CSS for equation \eqref{eq_FBP1_1} in $D$.
\end{proposition}

\begin{proof}
Consider a continuation of the coefficients $p$, $q$, $w$ onto $[-L_1, L]$ such that the Liouville transformation satisfies $l(-L_1) = l(L)=b$ and Assumption \ref{Assum_pqw} holds on $[-L_1,L]$.

Let $u(y,t)\in C^{2,1}(D)$ be a real valued solution to \eqref{eq_FBP1_1}, $K\subset D$ a compact subset and $\varepsilon>0$. Consider the preimage $u_l=l^{-1}(u)$ of the solution $u$ under the Liouville transformation. Let $K_l = l^{-1}(K)$.
Then there exist a constant $\delta>0$ and functions $s_1(t)$ and $s_2(t)$, analytic on a disk in the complex plane containing the segment $[\delta, T]$ such that the domain $D(s_1,s_2) = \{
(x,t): s_1(t)\le x\le s_2(t), \, t\in[\delta, T]\}$ satisfies
\[
K_l\subset D(s_1,s_2) \subset [0,b]\times[0,T].
\]
The solution $u_l$ is a classical solution  of the Liouville transformed parabolic equation in $D(s_1,s_2)$, continuous in $\bar D(s_1, s_2)$. Similarly to the proofs of Theorem 2.3.2 and 2.3.3 from \cite{colton1976solution} $u_l$ can be extended to the solution of the same equation on the rectangle $[-b,b]\times[\delta,T]$, and its Liouville transformation (which we denote by $\tilde u$) is then a solution of \eqref{eq_FBP1_1} on $[-L_1,L]\times [\delta, T]$.

Consider $v=\mathbf{T}^{-1}\tilde u$. Then $v$ is a solution of the heat equation on $[-b,b]\times[\delta,T]$. Since the system $\{v_{n}\}_{n\in
\mathbb{N}}$ is a CSS for the heat equation on the region $[-b,b]\times[\delta,T]$, there exist a constant $N$ and such constants $a_0,\ldots,a_N$ that
\[
\max_{(x,t)\in [-b,b]\times[\delta,T]}\left\vert
v(x,t)-\sum_{n=0}^{N}a_{n}v_{n}(x,t)\right\vert <\frac{\varepsilon}{\|\mathbf{T}\|}.
\]
Hence
\[
\begin{split}
\max_{(y,t)\in [-L_1,L]\times[\delta,T]}\left\vert
\tilde u(y,t)-\sum_{n=0}^{N}a_{n}u_{n}(y,t)\right\vert
& =\max_{(y,t)\in [-L_1,L]\times[\delta,T]}\left\vert
\mathbf{T}v(y,t)-\sum_{n=0}^{N}a_{n}\mathbf{T}v_{n}(y,t)\right\vert\\
&<\frac{\varepsilon}{\|\mathbf{T}\|}\cdot \|\mathbf{T}\| = \varepsilon.
\end{split}
\]
Now the proof follows observing that $K\subset [-L_1,L]\times[\delta,T]$.
\end{proof}

\begin{remark}
Note that the transmuted CSS defined by \eqref{Transmuted CSS} does not depend on a continuation of the coefficients $p$, $q$ and $w$
\end{remark}

\begin{remark}
The technique developed in \cite{colton1976solution} and used in the proof of Proposition \ref{Prop Complete system_1} requires the boundaries $y_{1,2}$ of the region to be separated, i.e., $y_1(t)<y_2(t)$, $t\in [0,T]$ and thus does not allow us to work directly with an approximation to the solution of the original problem in which $y_1(0)=y_2(0)$. For that reason we have to consider the intervals $[\delta, T]$.
\end{remark}

The idea to use the transmutation operator to transmute the CSS for the construction of the solutions was studied in the monographs \cite{colton1976solution, colton1980analytic}. At the time, the representation \eqref{eq_Txn} for the transmuted powers and the
representations of the next section were unknown, which limited the practical application of Colton's theory.

\subsection{Transmutation of the exponential CSS}

In \cite{KravchenkoNavarroTorba2017} a representation for the solutions to equation
\begin{equation*}
\mathbf{C}u=\omega ^{2}u,
\end{equation*}
was obtained in terms of Neumann series of Bessel functions. This representation can
be used to construct a CSS for equation \eqref{eq_FBP1_1}. Consider the set of
functions $\{e_{n}^{\pm }\}_{n\in \mathbb{N}}$ defined in \eqref{eq_en_trig}
where $\omega _{n}$ are chosen such that the limit \eqref{eq_cond_w_conv}
exists. The next proposition guarantees that it is in fact the CSS.

Let $D = \{(x,t): s_1(t)<x<s_2(t),\, 0<t<t_0\}$, where $s_1$ and $s_2$ are analytic functions of $t$ for $0\le t\le t_0$ and $s_1(t)<s_2(t)$ for $0\le t\le t_0$.

\begin{proposition}[{\citet[Cor. 5.4]{colton1980analytic}}]
\label{Prop_En} Let $h\in C^{2,1}\left(
D\right) \cap C(\bar{D})$ be a solution to the heat equation \eqref{eq_heatEq} in $D$. Then there exists an integer $N$ and constants $a_{0}^{\pm
},...,a_{N}^{\pm }$ such that
\begin{equation*}
\max_{\bar{D}}\left\vert h(x,t)-\sum\limits_{n=0}^{N}a_{n}^{\pm }e_{n}^{\pm
}(x,t)\right\vert <\varepsilon .
\end{equation*}
\end{proposition}

Since under the change of the variable $t\mapsto t+\delta$ each function $e_n^{\pm}$ remains the same up to a multiplicative constant, the system $\{e_{n}^{\pm }\}_{n\in \mathbb{N}}$ is the CSS in the sense required for Proposition \ref{Prop Complete system_1}.

Each of the basis functions $e_{n}$ is a solution to the heat equation \eqref
{eq_heatEq}. We define the transmuted basis functions as follows
\begin{equation*}
E_{n}^{\pm }\left( y,t\right) =\mathbf{T}[e_{n}^{\pm }\left( x,t\right)
]=e^{-\omega _{n}^{2}t}\mathbf{T}[e^{\pm i\omega _{n}x}].
\end{equation*}
Application of Theorem \ref{Teor_Phi_x} guarantees us that they are
solutions to equation \eqref{eq_FBP1_1}, i.e. $(\mathbf{C}-\partial
_{t})E_{n}^{\pm }=(\mathbf{CT}e_{n}^{\pm }-\partial _{t}\mathbf{T}e_{n}^{\pm
})=\mathbf{T}(\partial _{xx}-\partial _{t})e_{n}^{\pm }=0$ and the
application of Proposition \ref{Prop Complete system_1} guarantees that they form a CSS
for equation \eqref{eq_FBP1_1} on any compact contained in $[0,L]\times (0,T]$.

For the construction of functions $E_{n}^{\pm }$ we can use the explicit form
of the transmuted solutions $\mathbf{T}[\cos (\omega x)]$ and $\mathbf{T}%
[\sin (\omega x)]$, since
\begin{equation*}
\mathbf{T}[e^{\pm i\omega _{n}x}]=\mathbf{T}[\cos (\omega _{n}x)]\pm i%
\mathbf{T}[\sin (\omega _{n}x)],
\end{equation*}
presented in \cite{KravchenkoTorba_Neumann_2018}.

\subsection{Representation of the transmuted Sine and Cosine}

Two linearly independent solutions of equation
\begin{equation}
\mathbf{C}u=\omega ^{2}u  \label{eq_Cu_wu}
\end{equation}
can be obtained as images of $\cos\omega x$ and $\sin\omega x$, linearly independent solutions of the equation $z''=\omega^2 z$, under the action of the transmutation operator $\mathbf{T}$, and
will be denoted by
\begin{equation}
c( \omega ,y) =\mathbf{T}[\cos (\omega x)],\qquad \text{with }\quad
c( \omega ,0) =1/\rho (0)\quad\text{ and }\quad c^{\prime }\left( \omega
,0\right) =\tilde h,  \label{eq_c_def}
\end{equation}
and
\begin{equation}
s( \omega ,y)  =\mathbf{T}[\sin (\omega x)],\qquad \text{with }\quad
s( \omega ,0) =0\quad\text{ and }\quad s^{\prime }\left( \omega ,0\right) =%
\frac{\omega }{\rho (0)}\sqrt{\frac{w(0)}{p(0)}},  \label{eq_s_def}
\end{equation}
where
\begin{equation*}
\tilde{h}=\sqrt{\frac{\rho \left( 0\right) }{w\left( 0\right) }}\left( \frac{%
f^{\prime }\left( 0\right) }{f\left( 0\right) }+\frac{\rho ^{\prime }\left(
0\right) }{\rho \left( 0\right) }\right)   
\end{equation*}
and $f$ is a solution of \eqref{eq_PQR_1} that satisfies \eqref{eq_PQR_2} and appears in Theorem \ref{Teor_Phi_x}.

\begin{theorem}[{\citet[Theorem 4.1]{KravchenkoTorba_Neumann_2018}}]
\label{Teor_cs_pqr} Let the functions $p$, $q$ and $w$ satisfy the conditions from the Assumption \ref{Assum_pqw} and $f$ be the solution of \eqref{eq_PQR_1} satisfying \eqref{eq_PQR_2} and such that $f\neq 0$ for all $y\in \lbrack 0,L]$. Then two linearly independent solutions $c$ and $s$ of equation \eqref{eq_Cu_wu} for $\omega \neq 0$ can be written in the form
\begin{equation}
c(\omega ,y)=\frac{\cos (\omega l(y))}{\rho (y)}+2\sum_{m=0}^{\infty
}(-1)^{m}\alpha _{2m}(y)j_{2m}(\omega l(y))  \label{eq_cwy}
\end{equation}%
and%
\begin{equation}
s(\omega ,y)=\frac{\sin (\omega l(y))}{\rho (y)}+2\sum_{m=0}^{\infty
}(-1)^{m}\alpha _{2m+1}(y)j_{2m+1}(\omega l(y))  \label{eq_swy},
\end{equation}%
where $j_k$ stands for the spherical Bessel function of order $k$,
\begin{equation*}
l(y):=\int\nolimits_{0}^{y}\left( w(s)/p(s)\right) ^{1/2}ds,
\end{equation*}%
with the coefficients defined by
\begin{equation}
\alpha _{m}(y)=\frac{2n+1}{2}\left( \sum_{k=0}^{m}\frac{l_{k,m}\Phi _{k}(y)}{%
l^{k}(y)}-\frac{1}{\rho (y)}\right) ,  \label{eq_direct_alfa}
\end{equation}
where $\Phi _{k}$ are taken from Definition \ref{Def_Phik},  and $l_{k,m}$ is the coefficient of $x^k$ in the Legendre polynomial of order $m$. The solutions $c$ and $s$ satisfy the initial conditions \eqref{eq_c_def} and \eqref{eq_s_def}. The series in \eqref{eq_cwy} and \eqref{eq_swy} converge uniformly with respect to $y$ on $[0,L]$
and converge uniformly with respect to $\omega $ on any compact subset of
the complex plane of the variable $\omega $. Moreover, for the functions%
\begin{equation*}
c^{M}(\omega ,y)=\frac{\cos (\omega l(y))}{\rho (y)}+2%
\sum_{m=0}^{[M/2]}(-1)^{m}\alpha _{2m}(y)j_{2m}(\omega l(y))
\end{equation*}
and
\begin{equation*}
s^{M}(\omega ,y)=\frac{\sin (\omega l(y))}{\rho (y)}+2%
\sum_{m=0}^{[(M-1)/2]}(-1)^{m}\alpha _{2m+1}(y)j_{2m+1}(\omega l(y))
\end{equation*}%
the following estimates hold
\begin{align*}
\left\vert c(\omega ,y)-c^{M}(\omega ,y)\right\vert & \leq \sqrt{2l(y)}%
\varepsilon _{M}(l(y))\max_{y\in \lbrack 0,L]}\frac{1}{\left\vert \rho
(y)\right\vert }, \\
\left\vert s(\omega ,y)-s^{M}(\omega ,y)\right\vert & \leq \sqrt{2l(y)}%
\varepsilon _{M}(l(y))\max_{y\in \lbrack 0,L]}\frac{1}{\left\vert \rho
(y)\right\vert }
\end{align*}%
for any $\omega \in \mathbb{R}$, $\omega \neq 0$, and
\begin{align*}
\left\vert c(\omega ,y)-c^{M}(\omega ,y)\right\vert & \leq \varepsilon
_{M}(l(y))\sqrt{\frac{\sinh (2Cl(y))}{C}}\max_{y\in \lbrack 0,L]}\frac{1}{%
\left\vert \rho (y)\right\vert }, \\
\left\vert s(\omega ,y)-s^{M}(\omega ,y)\right\vert & \leq \varepsilon
_{M}(l(y))\sqrt{\frac{\sinh (2Cl(y))}{C}}\max_{y\in \lbrack 0,L]}\frac{1}{%
\left\vert \rho (y)\right\vert }
\end{align*}%
for any $\omega \in \mathbb{C}$, $\omega \neq 0$ belonging to the strip $ \abs{\operatorname{Im}\omega} \leq C$, $C\geq 0$, where
$\varepsilon _{M}$ is a function satisfying $\varepsilon _{M}\rightarrow 0$,
as $M\rightarrow \infty $,  uniformly with respect to $y\in [0,L]$.
\end{theorem}

\begin{remark}\label{Rem Solutions w 0}
For $\omega=0$ the two linearly independent solutions can
be represented as
\begin{align*}
c(0,y)& =\mathbf{T}\left[ 1\right]=f(y) , \\
\tilde{s}(0,y)& =\lim_{\omega \rightarrow 0}\mathbf{T}\left[ \frac{\sin (\omega x)}{%
\omega }\right] =\mathbf{T}\left[ x\right] =\Phi _{1}(y).
\end{align*}%
\end{remark}

We also have the representation for the derivatives of the solutions in \cite[Section 5]{KravchenkoTorba_Neumann_2018},
\begin{align*}
c^{\prime}(\omega,y) & =\sqrt{\frac{w(y)}{p(y)}}\left[ \frac{1}{\rho (y)}%
(G_{1}(y)\cos(\omega l(y))-\omega\sin(\omega
l(y)))+2\sum\limits_{m=0}^{\infty}(-1)^{m}\mu_{2m}(y)j_{2m}(\omega l(y))%
\right] \\
&\quad -\frac{\rho^{\prime}(y)}{\rho(y)}c(\omega,y)
\end{align*}
and%
\begin{align*}
s^{\prime}(\omega,y) & =\sqrt{\frac{w(y)}{p(y)}}\left[ \frac{1}{\rho (y)}%
(G_{2}(y)\sin(\omega l(y))+\omega\cos(\omega
l(y)))+2\sum\limits_{m=0}^{\infty}(-1)^{m}\mu_{2m+1}(y)j_{2m+1}(\omega l(y))%
\right] \\
&\quad -\frac{\rho^{\prime}(y)}{\rho(y)}s(\omega,y),
\end{align*}
where
\begin{equation*}
G_{1}\left( y\right) = G_{2}\left( y\right) + \tilde h, \qquad G_{2}\left( y\right)  =\frac{\rho \rho ^{\prime }}{2w}\bigg|_{0}^{y}+\frac{1}{2}\int_{0}^{y}
\left[ \frac{q}{\rho ^{2}}+\frac{\left( \rho ^{\prime }\right) ^{2}}{w}
\right] \left( s\right) ds,  
\end{equation*}
and
\begin{equation}\label{eq_direct_mu}
\begin{split}
\mu_{m}(y) & :=\frac{2m+1}{2\rho(y)}\biggl[\sum\limits_{k=0}^{m}\frac{l_{k,m}}{%
l^{k}(y)}\left( k\frac{\Psi_{k-1}(y)}{\rho(y)}+\rho(y)\sqrt{\frac {p(y)}{w(y)%
}}\left( \frac{f^{\prime}(y)}{f(y)}+\frac{\rho^{\prime}(y)}{\rho(y)}\right)
\Phi_{k}(y)\right)  \\
&\quad  -\frac{m(m+1)}{2l(y)}-G_{2}(y)-\frac{\tilde h}{2}(1+(-1)^{n})\biggr].
\end{split}
\end{equation}

We can
use Theorem \ref{Teor_cs_pqr} to represent the transmuted base functions and
their derivatives as follows
\begin{align}
E_{n}^{\pm}(y,t) & =\left( c\left( \omega_{n},y\right) \pm is(\omega
_{n},y)\right) e^{-\omega_{n}^{2}t},  \label{eq_En1} \\
\partial_{y}\left( E_{n}^{\pm}(y,t)\right) & =\left( c^{\prime}\left(
\omega_{n},y\right) \pm is^{\prime}(\omega_{n},y)\right)
e^{-\omega_{n}^{2}t},  \label{eq_En2} \\
\partial_{t}\left( E_{n}^{\pm}(y,t)\right) & =-\omega_{n}^{2}\left( c\left(
\omega_{n},y\right) \pm is(\omega_{n},y)\right) e^{-\omega_{n}^{2}t}.
\label{eq_En3}
\end{align}

\subsection{Recurrence formulas}\label{Subsect Recurrent}

The representations \eqref{eq_direct_alfa} and \eqref{eq_direct_mu} are not
practical for efficient computation of a large number of the coefficients
due to the fast growth of the Legendre coefficients $l_{k,m}$ when $m\to\infty$. An alternative, robust
for the computations recurrence formulas, were developed in \cite
{KravchenkoTorba_Neumann_2018}. We introduce
\begin{equation}
A_{n}\left( y\right) =l^{n}\left( y\right) \alpha _{n}\left( y\right) \qquad\text{and}\qquad B_{n}\left( y\right) =l^{n}\left( y\right) \mu _{n}\left(
y\right) ,  \label{NSBF_DBKO_eq10_1}
\end{equation}%
and then the following formulas hold for $n=2,3,...$
\begin{equation}\label{An recurr}
A_{n}\left( y\right) =\frac{2n+1}{2n-3}\left( l^{2}\left( y\right)
A_{n-2}\left( y\right) +\left( 2n-1\right) f\left( y\right) \tilde{\theta}%
_{n}\left( y\right) \right)
\end{equation}
and
\begin{equation}\label{Bn recurr}
\begin{split}
B_{n}(y)& =\frac{2n+1}{2n-3}\Biggl[l^{2}(y)B_{n-2}(y)+2(2n-1)\Bigg(\sqrt{%
\frac{p(y)}{w(y)}}\left( f^{\prime }(y)\rho (y)+f(y)\rho ^{\prime
}(y)\right) \frac{\tilde{\theta}_{n}(y)}{\rho (y)}  \\
&\quad  +\frac{\tilde{\eta}_{n}(y)}{\rho ^{2}(y)f(y)}\Bigg)-(2n-1)l(y)A_{n-2}(y)%
\Biggr],
\end{split}
\end{equation}
where
\begin{equation*}
\tilde{\theta}_{n}\left( y\right) =\int_{0}^{y}\left( \frac{\tilde{\eta}%
_{n}\left( x\right) }{\rho ^{2}\left( x\right) f^{2}\left( x\right) }-\frac{%
l\left( x\right) A_{n-2}\left( x\right) }{f\left( x\right) }\right) \sqrt{%
\frac{w\left( x\right) }{p\left( x\right) }}dx
\end{equation*}
and
\begin{equation*}
\tilde{\eta}_{n}(y)=\int_{0}^{y}\Bigg(l(x)(f^{\prime }(x)\rho (x)+f(x)\rho
^{\prime }(x))+(n-1)\rho (x)f(x)\sqrt{\frac{w(x)}{p(x)}}\Bigg)\rho
(x)A_{n-2}(x)dx.
\end{equation*}
The initial values $A_{0}$, $A_{1}$, $B_{0}$ and $B_{1}$ can be calculated
from
\begin{equation*}
A_0\left( y\right) = \frac{1}{2}\left( f\left( y\right) -\frac{1}{%
\rho \left( y\right) }\right) , \qquad
A_{1}\left( y\right) =\frac{3}{2}\left( \Phi _{1}\left(
y\right) -\frac{l\left( y\right)}{\rho \left( y\right)  }\right),
\end{equation*}
and
\begin{align*}
B_{0}\left( y\right) & =\sqrt{\frac{p\left( y\right) }{w\left( y\right) }}%
\left( f^{\prime }\left( y\right) +\frac{f\left(
y\right)\rho ^{\prime }\left(
y\right) }{\rho \left( y\right) } \right) -\frac{%
G_{1}\left( y\right) }{2\rho \left( y\right) }, \\
B_{1}\left( y\right) & =\frac 3{2} \left[\frac 1{f\left( y\right)\rho^2\left( y\right)}+ \sqrt{\frac{p\left( y\right) }{w\left( y\right) }}\left( \frac{\rho^{\prime }\left( y\right)}{\rho\left( y\right)} + \frac{f^{\prime }\left( y\right)}{f\left( y\right)}\right)\Phi_1\left( y\right)  - \frac{G_{2}\left(
y\right)l\left(
y\right)+1}{\rho\left( y\right)}
\right].
\end{align*}%
For the discussion on the computational details see \cite{KravchenkoTorba_Neumann_2018} and \cite{KravchenkoNavarroTorba2017}.

\subsection{Reduced complete system of solutions}\label{Subsect ReducedSystem}
Let us consider and important particular case when in \eqref{eq_FBP_b2_1} one has
\begin{equation}\label{eq_FBP_b2_constants}
    \gamma_{11}(t)\equiv \alpha,\qquad \gamma_{12}(t)\equiv \beta\ne 0,\qquad\text{and}\qquad g_1(t)\equiv 0,\qquad t\in (0,T),
\end{equation}
where $\alpha$ and $\beta$ are some real constants. We are going to show that the CSS $\{E_n^{\pm}\}$ can be reduced into one that a priori satisfies condition \eqref{eq_FBP_b2_1}. The reduced system is comprised of functions $\tilde E_n$ of the form 
\begin{equation}\label{Etilde n}
    \tilde E_n(y,t) = e^{-\omega_n^2 t} \bigl(c(\omega_n, y) + \beta_n s(\omega_n,y)\bigr),
\end{equation}
where the constants $\beta_n$ are such that condition \eqref{eq_FBP_b2_1} is fulfilled. Note that for $\omega_n\ne 0$ each $\tilde E_n$ is a linear combination of the functions $E_n^+$ and $E_n^-$. Using \eqref{eq_c_def}, \eqref{eq_s_def} and \eqref{eq_FBP_b2_constants}, condition \eqref{eq_FBP_b2_1} reduces to
\begin{equation*}
    \frac{\alpha}{\rho(0)} + \beta \tilde h +\beta\beta_n \frac{\omega_n}{\rho(0)}\sqrt{\frac{w(0)}{p(0)}} = 0,
\end{equation*}
hence
\begin{equation}\label{beta_n def}
    \beta_n = -\frac{1}{\omega_n} \sqrt{\frac{p(0)}{w(0)}}\left(\frac{\alpha}{\beta} + \tilde h \rho(0)\right).
\end{equation}
If one of $\omega_n$ is equal to $0$, say for certainty that $\omega_0=0$, only one function corresponds to it, which need not to satisfy condition \eqref{eq_FBP_b2_1}. Note that both linearly independent solutions corresponding to $\omega_0 =0$ are constructed as part of the procedure described in Subsection \ref{Subsect Recurrent}, see also Remark \ref{Rem Solutions w 0}, so we may take a linear combination of them
\begin{equation}\label{Etilde 0}
\tilde E_0(y,t) = f(y)+\beta_0 \Phi_1(y).
\end{equation}
 Condition \eqref{eq_FBP_b2_1} for $\omega_0 =0$ takes the form
\[
\frac{\alpha}{\rho(0)} + \beta\tilde h + \frac{\beta\beta_0}{\rho_0}\sqrt{\frac{w(0)}{p(0)}}=0,
\]
hence
\begin{equation}\label{beta_0 def}
\beta_0 = -\sqrt{\frac{p(0)}{w(0)}}\left(\frac\alpha\beta + \tilde h \rho(0)\right).
\end{equation}

Let us show that the system $\{\tilde E_n\}$ is a CSS.
Let $D=\{(y,t):0<y<y_2(t),\ t\in(0,T]\}$, where $0<y_2(t)\le L$ is a continuous functions, be a subset of $\mathbb{R}^2$. 

\begin{proposition}\label{Prop Reduced system}
Let the limit condition \eqref{eq_cond_w_conv} holds. Then the system of functions $\{\tilde E_n\}$ is a CSS for solutions of equation \eqref{eq_FBP1_1} in $D$ satisfying condition \eqref{eq_FBP_b2_1}.
\end{proposition}

\begin{proof}
We recall that for even functions $h$ the transmutation operator from Theorem \ref{Teor_Phi_x} can be written as
\[
\mathbf{T}h(y)=\frac{h(l(y))}{\rho (y)}+\int_{0}^{l(y)}\bigl(K(y,t)+K(y,-t)\bigr)h(t)dt,
\]
and in such form it acts from $C[0,b]$ to $C[A_0,B]$. It is a transmutation operator in the sense of Definition \ref{Def Transmutation} if we consider it as the operator $\mathbf{T}:E_1\to E_2$, where 
\begin{align*}
    E_1 &= \{h\in C^1[0,b]: h'(0)=0\},\\
    E_2 &= \{g\in C^1[A_0,B]: g'(A_0) - \rho(A_0)f'(A_0)\cdot g(A_0)=0\}.
\end{align*}
Moreover, a whole family of such transmutation operators exists, where the space $E_2$ can be changed to an arbitrary space 
\[
E_{2,\gamma} = \{ g\in C^1[A_0,B]:g'(A_0) - \gamma \cdot g(A_0)=0\}.
\]
Let us consider $\gamma = -\frac\alpha\beta$ and the corresponding transmutation operator $\mathbf{T}_\gamma$. 

Let $u$ be a solution of equation \eqref{eq_FBP1_1} in $D$ satisfying condition \eqref{eq_FBP_b2_1}. Since for any fixed $t_0>0$ one has $u(y,t_0)\in E_{2,\gamma}$, we can consider 
\[
h = \mathbf{T}^{-1}_\gamma[u],
\]
where the transmutation operator acts with respect to the first variable. Then $h$ is a solution of the heat equation in $l^{-1}(D)$ satisfying $h_x(0,t)=0$ for $t\in (0,T]$, c.f., proof of Proposition \ref{Prop Complete system_1}. Consider its continuation $h^e$ onto $D_e = \{ (x,t): -l^{-1}(y_2(t)) < x < l^{-1}(y_2(t)),\ t\in (0,T]\}$ as an even function of the variable $x$. One can check that $h^e$ is a solution of the heat equation in $D_e$, c.f., \cite[Thm. 4.6]{colton1980analytic}.

Let $\varepsilon>0$ and a compact $K\subset D$ be fixed. Then there exist $\delta>0$ and a compact set $K_1=\{ (x,t): 0\le x\le s(t),\ t\in [\delta,T]\}$, with analytic boundary $s$, such that $l^{-1}(K)\subset K_1\subset D_e$. Consider $K_2 = \{ (x,t): -s(t) \le x\le s(t),\ t\in[\delta, T]\}$. By Proposition \ref{Prop_En} there exist $N$ and constants $a_0^{\pm},\ldots,a_N^{\pm}$ such that
\[
\max_{K_2}\left| h^e(x,t) -\sum_{n=0}^N a_n^{\pm}e_n^{\pm}(x,t)\right|<\varepsilon.
\]
Denote $h_N(x,t):=\sum_{n=0}^N a_n^{\pm}e_n^{\pm}(x,t)$ and consider 
\[
h_N^e = \sum_{n=0}^N (a_n^+ + a_n^-)e^{-\omega_n^2 t}\cos(\omega_n x) = \frac{h_N(x,t) + h_N(-x,t)}2.
\]
Then 
\begin{align*}
\max_{K_1}|h^e(x,t) - h_N^e(x,t)| &= \frac 12 \max_{K_1}| h^e(x,t) + h^e(-x,t) - h_N(x,t) - h_N(-x,t)|\\
&\le \frac 12 \left(\max_{K_1}| h^e(x,t)- h_N(x,t)| + \max_{K}| h^e(-x,t)- h_N(-x,t)|\right)<\varepsilon.
\end{align*}

Now note that each term of $h_N^e$ is an even function, hence belongs to the space $E_1$, and $\mathbf{T}_\gamma[e^{-\omega_n^2 t}\cos(\omega_n x)] = \tilde E_n(y,t)$. Hence $\mathbf{T}_\gamma[h_N^e]$ is the sought for approximation to the solution $u$.
\end{proof}

\section{\label{sec_MinProblem}Minimization problem}

In this section we describe the scheme of the numerical method
proposed. In the previous section, we saw that any solution to the PDE \eqref%
{eq_FBP1_1} can be approximated by a linear combination of functions from the CSS of transmuted exponential functions. We denote by $u^{N}$ this approximation and by $a_{n}$, $n\in\{0,\ldots,N\}$ the
respective coefficients---see equation \eqref{eq_uN_H}. Note that we reordered the set of the functions $E_{n}^{\pm}(y,t)$ into the sequence $\{u_n(y,t)\}_{n=0}^\infty$ by setting, e.g., $u_{2n} = E_{n}^{+}$ and $u_{2n+1} = E_{n}^{-}$.  We also denote by $\bar{t}=(t_{0},...,t_{N_{t}})$ an ordered numerical set of $N_{t}+1$ points on the interval $[0,T]$, with $t_{0}=0<t_{1}<...<t_{N_{t}}=T$. Similarly, we construct the vector $\bar{y}=(y_{0},...,y_{N_{y}})$, on an interval $[y_{0},y_{N_{y}}]$, the bounds will be specified further. We look for the free boundary in the form
\begin{equation}
s_{K}(t)=\sum\limits_{k=0}^{K}b_{k}\beta _{k}(t),  \label{eq_RepFB}
\end{equation}%
where $\beta _{k}:[0,T]\rightarrow \mathbb{R}$, $k=0,1...,K$ is a set of $%
K+1 $ linearly independent functions.\footnote{ We can choose a more general representation for the boundary if needed. See \cite{KrKrTorba2017} for the discussion.}

Recall that any expression of the form \eqref{eq_uN_H} is a solution to \eqref{eq_FBP1_1}. Hence, our problem now reduces to finding the coefficients $\bar{a}=(a_{0},...,a_{N})$ for the approximate solution and $\bar{b}=(b_{0},...,b_{K})$ for the free boundary in such a way that the approximate solution is close to the exact solution of Problem \ref{Problem_FBP_1}. For this purpose, according to Assumption \ref{Continuous depSol}, it is sufficient to minimize the discrepancy for the boundary conditions \eqref{eq_FBP_b2_1}--\eqref{eq_FBP_b4_1} in a suitable $L_2$ norm. We consider the following one for each boundary condition
\begin{equation}\label{Formula for norm}
\left\Vert v(\bar{t})\right\Vert ^{2}=\left\Vert
(v(t_{0}),...,v(t_{N_{t}}))\right\Vert ^{2}=\sum_{i=0}^{N_{t}}\vphantom{\sum|}''\left\vert
v(t_{i})\right\vert ^{2},
\end{equation}
where the double prime indicates
that the first and the last terms of the sum are to be halved.
This formula is the discrete approximation for the $L_2$ norm on the segment $[0,T]$, and for different choices of the points $t_k$ reduces either to trapezoidal rule (for uniformly distributed points $t_k$) or to the highly accurate Lobatto--Tchebyshev integration rule of the first kind (for $t_k$ being Tchebyshev nodes), see \cite[(2.7.1.14)]{DavisRabinowitz1984}. With this representation, the minimization problem that we have to solve takes the following form.
\begin{problem}
\label{Problem_ArgMin_A}Find \footnote{%
For a function $f:X\rightarrow Y$, the $\arg\min$ over a subset $S$ of $X$
is defined as%
\begin{equation*}
\underset{x\in S\subseteq X}{\arg\min}\,f(x):=\left\{ x:x\in S\wedge\forall
y\in S:f(y)\geq f(x)\right\} .
\end{equation*}%
}
\begin{equation*}
\underset{\left( \bar{a},\bar{b}\right) }{\arg\min}\,F\left( \bar{a},\bar{b}%
\right) ,
\end{equation*}
subject to
\begin{equation}
s_K(0) = 0,\qquad 0<s_{K}(t)\leq L,\qquad t\in(0,T],  \label{eq:_probM_condS_1}
\end{equation}
where%
\begin{equation}
F\left( \bar{a},\bar{b}\right) =\sum_{i=1}^{3}I_{i}^{2}\left( \bar{a},\bar{b}%
\right)  \label{eq_F_1}
\end{equation}
and
\begin{align*}
I_{1}\left( \bar{a},\bar{b}\right) & =\left\Vert \gamma_{11}\left( \bar{t} \right) u^{N}\left( 0,\bar{t}\right) +
\gamma_{12}\left( \bar{t} \right) \left( u^{N} \right)_{y} \left( 0,\bar{t}\right)
 -g_{1}(\bar{t}) \right\Vert , \\
\displaybreak[2]
I_{2}\left( \bar{a},\bar{b}\right) &=\left\Vert u^{N}\bigl(s_{K}\left( \bar{t}%
\right) ,\bar{t}\bigr)-g_{2}(\bar{t})\right\Vert , \\
\displaybreak[2]
I_{3}\left( \bar{a},\bar{b}\right) & =\left\Vert \left( u^{N}\right) _{y}%
\bigl(s_{K}\left( \bar{t}\right) ,\bar{t}\bigr)-g_{3}(\bar {t})\right\Vert .
\end{align*}
\end{problem}
The value of the function $F$ indicates the discrepancy with the exact solution.
\begin{remark}
\label{Rem_Linearization}For fixed $\bar{b}$, the constrained Problem \ref%
{Problem_ArgMin_A} reduces to the unconstrained least squares minimization problem for
the coefficients $\bar{a}$ and can be solved exactly. That is, for each $\bar b$ we can define
\begin{equation}\label{Linear LS}
\bar a(\bar b) :=\underset{\bar{a}}{\arg\min}\,F\left(  \bar{a},\bar
{b}\right).
\end{equation}
So instead of minimizing the value function $F$ over an $N+K+2$ dimensional space of parameters $(\bar a, \bar b)$, the problem can be reduced to minimization of the function
\begin{equation}\label{eq_F_after_LS}
\tilde F(\bar b):= F\bigl(\bar a(\bar b), \bar b\bigr)
\end{equation}
over a $K+1$ dimensional space.
This reformulation of the problem leads to a more robust convergence of the numerical method--- see \cite{herreraporter2017}. We will apply this technique to the FHRO in Section \ref{sec_NumResults}---see also \cite{KrKrTorba2017} for details in the THP case.
\end{remark}

At this point, we can schematize the algorithm for constructing an approximate solution to Problem \ref{Problem_FBP_1} starting from the
exponential series \eqref{eq_en_trig} as a CSS for the heat equation and transmuting it to CSS for equation \eqref{eq_FBP1_1}.

\subsection{Conceptual algorithm}

\begin{enumerate}[label=(\roman*)]
\item Find a particular solution $f$ for the equation \eqref{eq_PQR_1} that
satisfies \eqref{eq_PQR_2}. The SPPS (Spectral parameter power series) method of \cite{KravchenkoPorter10} can be
used or any alternative analytical or numerical method.

\item Compute the coefficients $\alpha_{n}$ and $\mu_{n}$ using the
recursive formulas \eqref{NSBF_DBKO_eq10_1}, \eqref{An recurr} and \eqref{Bn recurr}.

\item Choose a sequence $\omega_{n}$ satisfying \eqref{eq_cond_w_conv} and construct the functions $E_{n}^{\pm}(y,t)$, $n=0,\ldots,N$
and their derivatives by formulas \eqref{eq_En1}--\eqref{eq_En3}.

\item Choose the basis functions $\beta_0,\ldots,\beta_{N_k}$ for the approximation of the free boundary function in the form  \eqref{eq_RepFB}.

\item Construct the minimization function $\tilde F$ from equation \eqref{eq_F_after_LS}.

\item Run a minimization algorithm for the function $\tilde F$ under constraints \eqref{eq:_probM_condS_1}.
\end{enumerate}

\begin{remark}\label{Rmk Reduced system}
In the particular case considered in Subsection \ref{Subsect ReducedSystem}, reduced CSS can be used. The changes to the proposed algorithm are minimal: we do not need to reorder the functions $\tilde E_n$, and the functional $I_1$ is always equal to zero. We left the remaining details to the reader.
\end{remark}

The application of the above schematics on the valuation of FHRO will be presented in the next sections.

\section{\label{sec_RO}The Russian option}

The FHRO is a theoretical path-dependent financial contract, a special case of an American lookback option. It was first introduced and studied in \cite%
{shepp1993russian,shepp1995new}. The owner of the Russian option has the right, but not the obligation, to exercise it any time and receive
the supremum of stock archived during the period between the writing of an option ($t=0$) and the exercise time. Originally, the Russian option was
defined as a perpetual option (infinite horizon $T=\infty $) of the \textquotedblleft reduced regret\textquotedblright---\cite{shepp1993russian} and \cite{duffie1993arbitrage}. The problem of pricing this option complicates if we want to treat finite horizon cases ($\infty>T>0 $).

The case where the underlying asset movement is given by the geometric Brownian motion, i.e. pricing under the BSM model, was widely studied. For the infinite horizon, there is a closed form solution, that for convenience of the reader is presented in the Appendix. For the finite horizon, the theoretical results can be consulted for instance in \cite{ekstrom2004russian}, \cite{peskir2005russian} and \cite{duistermaat2005finite}. The Bachelier model was analyzed in \cite{kamenov2008bachelier, kamenov2014Dissertation}. In the latest work some theoretical results for more general models are also presented.

The price of the option satisfies a certain free boundary problem for the parabolic PDE. For the BSM model there are several quantitative studies, e.g. \cite{duistermaat2005finite} by the method referred to as $n$th-order randomization, based on a method proposed by \cite{carr1998randomization} for American options, \cite{kimura2008valuing} applying the Laplace-Carlson transform and \cite{jeon2016integral} defining an equivalent PDE problem with mixed boundary conditions and solving it using Mellin transform. These methods rely on the possibility of explicit solving the respective transformed problems and hence are restricted to the BSM model.

\subsection{The set-up of the FBP for FHRO}

The value of the FHRO\ depends on three variables: price of the underlying asset ($s$), the maximum of the underlying asset ($m$) and time ($z$). As we will see
further, it can be reduced to the FBP with only two variables, due to the
homogeneity property of the value function. The definition of the problem
that we follow is from \citet[Theorem 1]{ekstrom2004russian} and \cite{kimura2008valuing}. An equivalent derivation can be consulted in \citet[
Theorem 3]{duistermaat2005finite}, \cite{peskir2005russian} and \citet[Section 26.2.5]{PeskirShiryaev06}.

Under the risk neutral measure the FHRO at the time $z\in \left[ 0,T\right] $, with $T>0$ being the time horizon of the option price, is given by
\begin{equation*}
V\left( s,m,z\right) =\underset{0\leq \theta _{z}\leq T-z}{\operatorname{ess~sup} }E_{s,m}%
\left[ e^{-r\theta _{z}}M_{\theta _{z}}\right] ,  
\end{equation*}%
where
\begin{equation*}
M_{z}=m\vee \sup_{0\leq u\leq z}S_{u},\qquad z\geq 0,
\end{equation*}%
is the supremum process,
\begin{equation*}
S_{z}=s\exp \left\{ \left( r-\delta -\frac{1}{2}\sigma_{0}^{2}\right)
z+\sigma_{0} B_{z}\right\} ,\qquad z\geq 0,
\end{equation*}%
is the price process for the underlying asset, with: $S_{0}=s$ -- the initial
fixed value; $r>0$ -- the risk free rate of interest; $\delta \geq 0$ --
the continuous dividend rate; $\sigma _{0}>0$ -- the volatility coefficient of
the asset price; $B_{z}$ -- the one-dimensional standard Brownian motion on a filtered probability space
$\left( \Omega ,\mathbb{F},\left( \mathcal{F}_{z}\right)_{z\geq 0},\mathbb{Q}\right)$; $\left( \mathcal{F}_{z}\right) _{z\geq 0}$ -- the filtration generated by $B_{z}$; $\mathbb{Q}$ -- the
probability measure chosen so that the stock has a mean of return $r$; $
\theta _{z}$ -- the stopping time of the filtration $\mathbb{F}$; $E_{s,m}
\left[ \cdot \right] \equiv E\left[ \cdot \mid \mathcal{F}_{0}\right] =E\left[ \cdot \mid
S_{0}=s,M_{0}=m\right] $ is calculated under the risk neutral measure $
\mathbb{Q}$. Also, we define the early exercise boundary
\begin{equation*}
S(m,z)=\inf \{s\in \lbrack 0,m]:(s,m,z)\in \mathcal{C}\},
\end{equation*}%
where $\mathcal{C=}\left\{ \left( s,m,z\right) :V\left( s,m,z\right)
>m\right\} $ is the so called continuation region. The function $S(m,z)$ is
non-decreasing and continuous in $z$ for $\delta >0$, see \cite[Theorem 2]{ekstrom2004russian} and \cite{duistermaat2005finite}).

\begin{theorem}[{\citet[Theorem 1]{ekstrom2004russian}}]
\label{Teor_EkstromV} The value of the FHRO is a solution $V(s,m,z)$ of the following free
boundary problem:%
\begin{equation*}
V_{z}+\frac{\sigma _{0}^{2}}{2}s^{2}V_{ss}+(r-\delta )sV_{s}-rV=0\qquad \text{for }\quad S(m,z)<s\leq m  
\end{equation*}%
with boundary conditions:
\begin{gather*}
V(s,m,z)=m\qquad \text{if }\quad S(m,z)\geq s,  
\\
\lim_{\varepsilon \rightarrow 0}\frac{1}{\varepsilon }(V(s,s+\varepsilon
,z)-V(s,s,z))=0,  
\\
V(s,s,z)=0\qquad\text{on }\quad S(m,z)=s,  
\\
V_{s}\left( s,m,z\right) \leq V(1,1,z),  
\\
V(s,m,T)=m.
\end{gather*}
\end{theorem}

The homogeneity of the function $V$, that is
\begin{equation*}
V(ks,km,z)=kV(s,m,z),\qquad \text{ for all }k\in \mathbb{R}^{+},  
\end{equation*}%
suggests that the problem is two dimensional. Consider the following change of the dependent variable
\begin{equation}
V(s,m,z)=mV\left( \frac{s}{m},1,z\right) =:mu\left( 1-y,t\right),\label{eq_Var_u}
\end{equation}%
where
\begin{equation}
y=1-s/m\qquad\text{and}\qquad t=T-z  \label{eq_Var_y_t}
\end{equation}%
are the new independent variables. Moreover, we also introduce the following notation for the free boundary
\begin{equation*}
b(t):=1-S(m,T-z)/m.  
\end{equation*}%
Then the FBP problem for the FHRO under the BSM model can be written as follows.

\begin{problem}
\label{Problem_KimuraFBPv}Find functions $u(y,t)$ and $b(t)$, such that
\begin{equation}\label{eq_DiffOp_Complete}
-u_{t}+\mathbf{M}u=0,\qquad\text{for \ }b\left( t\right) >y\geq0,t\in
\lbrack0,T],
\end{equation}
where
\begin{equation}
\mathbf{M}=\frac{1}{2}\sigma_{0}^{2}\left( 1-y\right) ^{2}\partial
_{yy}-\left( r-\delta\right) \left( 1-y\right) \partial_{y}-r,
\label{eq_DiffOp_M}
\end{equation}
and the boundary conditions%
\begin{align}
u\left( b\left( t\right) ,t\right) & =1,  \label{eq_FBP_a1} \\
u_{y}\left( b\left( t\right) ,t\right) & =0,  \label{eq_FBP_a2} \\
u\left( 0,t\right) +u_{y}\left( 0,t\right) & =0,  \label{eq_FBP_a3}\\
u_{y}(y,t)+u(0,t)&\geq0,  \label{eq_FBP_a5}\\
b\left( 0\right) &=0  \label{eq_FBP_a6}
\end{align}
are satisfied.
\end{problem}


Problem \ref{Problem_KimuraFBPv} compared to Problem \ref{Problem_FBP_1} has an additional condition \eqref{eq_FBP_a5}.
If one looks at the proof of Theorem 26.3 from \cite{PeskirShiryaev06}, especially at the part including formula (26.2.31), one can see that the condition \eqref{eq_FBP_a5} is used to deduce the monotonicity and finiteness of the boundary function $b(t)$. Having this property established, the corresponding FBP possesses a unique solution amongst monotone boundaries, see part 7 of the proof. So we can reformulate Problem \ref{Problem_KimuraFBPv} as follows, without additional conditions compared to Problem \ref{Problem_FBP_1}.

\begin{problem}
\label{Problem_KimuraFBPv2}Find functions $u(y,t)$ and $b(t)$, a monotone non-decreasing function, such that equation \eqref{eq_DiffOp_Complete}
and boundary conditions \eqref{eq_FBP_a1}, \eqref{eq_FBP_a2}, \eqref{eq_FBP_a3} and \eqref{eq_FBP_a6}
are satisfied.
\end{problem}

Problem \ref{Problem_KimuraFBPv2} has non-consistent boundary conditions, i.e., it is impossible to satisfy all the boundary conditions simultaneously at the point $(0,0)$. Indeed, conditions \eqref{eq_FBP_a1} and \eqref{eq_FBP_a2} imply $u(0,0)+u_y(0,0)=1$, a contradiction to the condition \eqref{eq_FBP_a3}. This observation already leads us to expect the computational difficulties near the origin.

We will refer to $u$ from Problem \ref{Problem_KimuraFBPv2} as value function and to $u(y,T)$ the option value, these are usually the functions studied in the literature, we can compute the value of the Russian option from these functions by the transformations \eqref{eq_Var_u} and \eqref{eq_Var_y_t}.
\begin{remark}
The classical transformation can be used to reduce the differential operator $\mathbf{M}$ from \eqref{eq_DiffOp_M} to $pqw$ form \eqref{eq_FBP1_1}---see e.g. \citet[Sections 0.4.1-3]{Polyanin2001}.
\end{remark}

\begin{remark}\label{rem_assymptBoundary} \emph{Theoretical results for the free boundary, asymptotics at the origin and the infinite horizon case.}
In the case of the infinite horizon (i.e., perpetual option) the problem can be solved exactly---see \cite{shepp1993russian, shepp1995new}. For the
sake of completeness we have included the solution in the Appendix. The infinite horizon is an important bound that we can use in the minimization process,
since we know that the value of the FHRO should be lower.

The free boundary can not have a smooth behaviour at the origin. This was confirmed by the theoretical result established in \cite{ekstrom2004russian}
and \cite{peskir2005russian}. The asymptotics as $t\rightarrow 0$ is given by
\begin{equation}
b(t)\sim \sigma _{0}\sqrt{t\left\vert \log (1/t)\right\vert }.
\label{eq_assymptotics_bt}
\end{equation}

\end{remark}

\section{\label{sec_NumResults}Numerical experiments}

In this section, we analyze the application of the proposed
algorithm as well as the arising numerical issues and their solutions. The results confirm the convergence of the method as well as some numerical values that appear in the bibliography for Problem \ref{Problem_KimuraFBPv2}.

\subsection{First steps}
For the implementation details of the first two steps of the proposed algorithm, i.e., construction of a particular solution $f$ and of the coefficients $\alpha_n$ and $\mu_n$ we refer the reader to \cite{KravchenkoNavarroTorba2017}, \cite{KravchenkoTorba_Neumann_2018}, \cite{KrKrTorba2017} and only want to mention that since the maximum upper boundary $b_{\infty }$ is known---see Appendix \ref{Sec_Ap_InfHor}, we only need values of ${E}^{\pm}_{n}(y)$ on the interval $[0,b_{\infty }]$. In our computations we have used this knowledge and chose the interval $[0,L]$ to be a bit larger than $[0,b_{\infty }]$. All the functions involved were represented by their values on 10000 points uniform mesh.

\subsection{The choice of $\{\protect\omega _{n}\}$}
The optimal choice for the set $\{\omega _{n}\}$ is an open question. Since the condition \eqref{eq_cond_w_conv} is for the convergence at infinity, we have total liberty for the choice of the first finite number of $\omega$'s. The only exception is that the pair of solutions for $\omega=0$ is constructed as a part of the representation for $c(\omega,y)$ and $s(\omega,y)$ from Theorem \ref{Teor_cs_pqr}, see Remark \ref{Rem Solutions w 0}. For this reason we always include $\omega_0=0$ in the set $\{\protect\omega _{n}\}$ and from now on we assume that $0=\omega_0<\omega_1<\ldots < \omega_N<\ldots$.

In the experiments, we used a pseudo-random algorithm to generate $\{\omega_{n}\}$ that depend on the set up step $d>0$ and density $\Delta$ and works as follows: it starts with $\omega_0=0$ and set $\omega_{n+1}=\omega_n+r_n+d$, where $r_n$ is a random number between 0 and $\Delta$. In our experience, too few leads to less accurate approximation, too many leads to functions linearly dependent up to machine error and hence the difficulty in solving the related linear problems. The upper bound for $\{\omega _{n}\}$ can be easily established: it is set where the value of $e^{-\omega _{n}^2T}$ becomes too small, (e.g., we have considered $\abs{\omega_{n}^2T}<100$). And we found that about $50-100$ values of $\omega_n$ allow us to obtain sufficiently accurate results, further increase in the number of $\omega_n$ does not lead to noticeable improvement.

This arbitrariness of the choice allows to test the algorithm under different choices of $\omega _{n}$, though its convergence to almost the same values is another confirmation of its robustness.

\subsection{Reduced system of solutions}
Since boundary condition \eqref{eq_FBP_a3} is of particular type considered in Subsection \ref{Subsect ReducedSystem}, we use reduced complete system of solutions $\{\tilde E_n\}$. Let us denote $\tilde E_n(y) = c(\omega_n,y) + \beta_n s(\omega_n,y)$. Then $\tilde E_{n}(y, t) = e^{-\omega_n^2 t} \tilde E_n(y)$. We approximate the value function by a truncated series
\begin{equation}
u_{N}(y,t)=\sum\limits_{n=0}^{N}a_{n}\tilde{E}_{n}(y)e^{-\omega _{n}^{2}t}.
\label{eq_uN_RO}
\end{equation}%

The computation of the value function \eqref{eq_F_1} requires the possibility to compute values of $\tilde {E}_{n}(y)$ at arbitrary point $y\in[0,L]$. For that we have approximated the functions $\tilde E_n(y)$ by splines using the routine \texttt{spapi} in Matlab.

\subsection{Representation of the free boundary}

The boundary asymptotics \eqref{eq_assymptotics_bt} presented in Remark \ref{rem_assymptBoundary} possesses factor $\sqrt t$ and unbounded derivative at $t=0$ suggesting that the polynomial approximation is not the best choice for the free boundary and that the following form
\begin{equation}
s_{K}(t)=\sqrt{t}\left( \sum\limits_{k=0}^{K}b_{k}t^{k/2}\right)
\label{eq_boundary_Ap}
\end{equation}
may be better. For faster convergence of the minimization  we have orthonormalized the set of functions $\{t^{k/2}\}_{k=1,...,K+1}$, using the $L^{2}(0,T)$ norm. We have for any polynomials $P_n$ and $P_m$
\[
\int_0^T \sqrt t P_n(\sqrt t) \cdot \sqrt t P_m(\sqrt t)\,dt = 2\int_0^{\sqrt{T}} t^3 P_n(t)P_m(t)\,dt.
\]
The orthogonal polynomials on the segment $[0,\sqrt{T}]$ with the weight $t^3$ coincide up to a multiplicative constant with the Jacoby polynomials $P_n^{(0,3)}\left(\frac{2t}{\sqrt T}-1\right)$, see \cite[(4.1.2)]{Szego75}. Hence using the formula (4.3.3) from \cite{Szego75} we obtain that the orthonormalized set consists of the functions
\[
\beta_k(t) =  \sqrt\frac{(k+2)t}{4T} P_k^{(0,3)}\left(2\sqrt\frac{t}{T}-1\right),\qquad k=0,\ldots, K.
\]
For the computations $K=9$ was used.

The grid $\bar{t}$ was taken to contain 2000 points and was selected to be less dense near $t=0$ (the problematic point) and more dense near $t=T$. For that we selected the points $t_n$
as a half of the Tchebyshev points, by the formula $t_{n}=T\sin (n\pi /(2N_{t}))$. The point $t_{0}=0$ was excluded due to inconsistency of the boundary conditions at this point. We would like to mention that the norm \eqref{Formula for norm} under such selection of the points $t_n$ can be reduced to Lobatto-Tchebyshev integration rule of the first kind, see \cite[(2.7.1.14)]{DavisRabinowitz1984}. We would also like to mention that the uniform distribution for the points $t_n$ worked almost equally well.

\subsection{Solution of the least squares minimization problem \eqref{Linear LS}}

For the fixed $
\overline{\tilde{b}}$, the minimization Problem \ref{Problem_ArgMin_A}
reduces to an unconstrained least squares minimization problem \eqref{Linear LS} that can be solved exactly. This solution will be denoted by $\tilde{a}$. It can be constructed as
follows. Under the notation
\begin{equation*}
\tilde{s}_{K}(t)=\sum\limits_{k=0}^{K}\tilde{b}_{k}\beta _{k}(t),
\end{equation*}%
for the free boundary with fixed coefficients $\overline{\tilde{b}}$, the
boundary conditions \eqref{eq_FBP_a1} and \eqref{eq_FBP_a2} take the form
\begin{align*}
\bar{1}& = u_{N}(\tilde{s}(\bar{t}),\bar{t})=\sum\limits_{n=0}^{N}\tilde{a}%
_{n}\tilde{E}_{n}(\tilde{s}(\bar{t}))e^{\omega _{n}\bar{t}}, \\
\bar{0}& = (u_{N})_{y}(\tilde{s}(\bar{t}),\bar{t})=\sum\limits_{n=0}^{N}%
\tilde{a}_{n}\tilde{E}_{n}^{\prime }(\tilde{s}(\bar{t}))e^{\omega _{n}\bar{t}%
}.
\end{align*}%
The relations for $\tilde{a}$ can be written in the matrix form as%
\begin{equation}
\mathbf{D}\tilde{a}=\mathbf{g},  \label{eq_linSys}
\end{equation}%
where%
\begin{equation*}
\mathbf{D}=\left[
\begin{array}{ccc}
\tilde{E}_{0}(\tilde{s}(\bar{t}))e^{\omega _{0}\bar{t}} & ... & \tilde{E}%
_{N}(\tilde{s}(\bar{t}))e^{\omega _{0}\bar{t}} \\
\tilde{E}_{0}^{\prime }(\tilde{s}(\bar{t}))e^{\omega _{0}\bar{t}} & ... &
\tilde{E}_{N}^{\prime }(\tilde{s}(\bar{t}))e^{\omega _{0}\bar{t}}%
\end{array}%
\right] \qquad \text{and}\qquad\mathbf{g}=\left(
\begin{array}{c}
\bar{1} \\
\bar{0}%
\end{array}%
\right).
\end{equation*}%
The solution of this overdetermined system coincides with the unique solution
of a fully determined one---see \citet[Theorem 5.14]{MadsenNielsen2010}, \cite{lawson1995solving} or \cite{wright2006numerical} for various methods of
solution. Note that the linear problem \eqref{eq_linSys} is
ill-conditioned, meanwhile is better than the one appearing in relation with the generalized heat polynomials, see \cite{KrKrTorba2017}. As a result, we were able to work with approximations \eqref{eq_uN_RO} containing as many as 100 functions $\tilde E_n$. However direct solution of the system \eqref{eq_linSys} results in large coefficients in the solution vector $\tilde a$ and hence in large round-off errors in the resulting approximate solution \eqref{eq_uN_RO}.
This can be easily solved by applying Tikhonov regularization to find a solution vector $\tilde a$ having relatively small coefficients. We have used the Matlab package \texttt{Regularization Tools} by Christian Hansen (see, e.g., \cite{hansen1994regularization}) to implement the regularization.

\subsection{Minimization process}
Minimization of the function $\tilde F$ from \eqref{eq_F_after_LS} was done with the help of \texttt{fmincon} function from Matlab. As the initial guess for the free boundary we took $s_K=c\beta_0$, where a constant $c$ was such that $s_K(T)<b_\infty$.

Two additional implementation details were somewhat unexpected to us however resulted in more robust convergence and lower resulting minimum value for the function $\tilde F$.

Fist, instead of minimizing the function $\tilde F$, we run the minimization process for the function $\sqrt{\tilde F}$. As a result, if in an experiment for the function $\tilde F$ the lowest value found by \texttt{fmincon} was $1.3\cdot 10^{-4}$, when applied to the function $\sqrt{\tilde F}$ the returned minimum value for the function $\tilde F$ was $5\cdot 10^{-9}$.

Second, the robustness of the minimization process as well as the returned minimal value may improve by posing additional constraints for the problem, letting somehow the function \texttt{fmincon} to avoid local minimums. The problem formulation possesses constraint \eqref{eq:_probM_condS_1} and additionally (see formulation of Problem \ref{Problem_KimuraFBPv2}) asks the free boundary to be monotone non-decreasing function, which can be written for our approximate boundary as
\begin{equation}\label{Constrain 2}
    s_K'(t)\ge 0, \qquad 0<t\le T.
\end{equation}
Additionally to these two natural constraints we considered the following one: we asked the free boundary to be a concave function, such form of the boundary can be see in \cite{kimura2008valuing}, \cite{jeon2016integral}. That is, in terms for our approximate boundary we posed additionally
\begin{equation}\label{Constrain 3}
    s_K''(t)\le 0, \qquad 0<t\le T.
\end{equation}
This additional constraint resulted to produce excellent results. For different choices of the exponents $\{\omega_k\}$ and different initial guesses for the free boundary, minimization process always converged to very close results. We have tried to improve the minimum by using returned vector $\bar b$ as an initial guess and running minimization process without additional constraint \eqref{Constrain 3} however with no success. Other standard ideas like to run the minimization process for a small $K$ and reuse the returned vector padded with zeros as an initial guess for larger $K$ do not produce significant improvements.

\subsection{Numerical results presentation}

There are several quantitative studies in the literature on the FHRO for the
BSM model. We will mainly compare our results with the Laplace--Carlson transform method (LCM) from \cite{kimura2008valuing}
for the long horizon and with the recursive integration method (RIM) from \cite{jeon2016integral} for the short horizon.\footnote{We would like to thank Junkee Jeon for providing us additional values that where not presented in their paper.} For the short horizon we have other values for the comparison, produced by the binomial tree model (BTM) and also reported in \cite{jeon2016integral}. We will refer as TES (transmuted exponential system) for the results produced by the proposed method

We start by presenting in Figure \ref{GraphValueFunction_1} the solution $u$, value option surface. As expected, it increases with time (recall that in our notation $t=0$ is the option expiry) and decreases with the initial value of variable $y$ (recall that $y=0$ corresponds to the initial value of the coefficient $s/m = 1$, i.e. the initial values of the option process and of the supremum process coincide). The condition \eqref{eq_FBP_a5} is satisfied. The cuts for the value of the option in time $T$, i.e. $(y,u(y,T))$ and the free boundary $(t,s_{K}(t))$ are presented in Figure \ref{Fig_Ex_FB_OptionValue}. We have chosen the following standard parameters for the model: $r=0.05$, $\delta =0.03$ and $\sigma _{0}=0.3$.

\begin{figure}[H]
\centering
  \includegraphics*[trim= 0 0 0 0,scale=1.00]{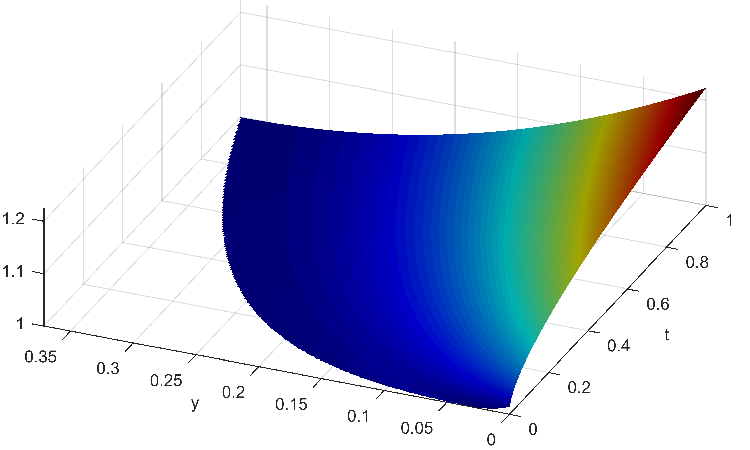}
  \caption{The value function for Problem \ref{Problem_KimuraFBPv2}, with parameters $r=0.05$, $\delta=0.03$, $\sigma_{0}=0.3$, $T=1$.} \label{GraphValueFunction_1}
\end{figure}


\begin{figure}[H]
  \centering
    \begin{tabular}{cc}
      \resizebox{3in}{!}{\includegraphics{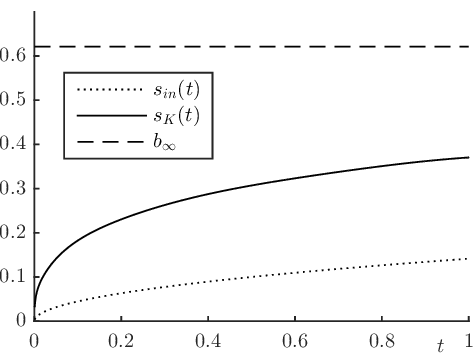}} &
      \resizebox{3in}{!}{\includegraphics{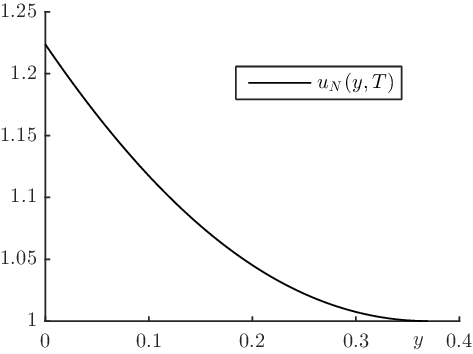}} \\
    \end{tabular}
    \caption{Left: the free boundary $s_{K}(t)$, the initial boundary $s_{in}(t)=0.1\beta_{0}(t)$ and the infinite horizon bound $b_{\infty}=0.6211$. Right: the value of the option, i.e. $u_{N}(y,T)$. Parameters (for both figures): $T=1$, $\sigma=0.3$, $\delta=0.03$, $r=0.05$, $N_t=2001$, $K=10$ and $\omega_{n}$ selected with a fixed step of $1/10$ and random step of $1/3$ (resulting in $N=68$). }
    \label{Fig_Ex_FB_OptionValue}
\end{figure}

In Figure \ref{Fig_Errors} the typical absolute errors that we obtain for the boundary conditions \eqref{eq_FBP_a1} and \eqref{eq_FBP_a2} are presented. Recall that condition \eqref{eq_FBP_a3} is satisfied by construction. In Figure \ref{Fig_Coefficients} the typical absolute values of the coefficients $\bar{a}$ and $\bar{b}$ obtained by solving Problem \ref{Problem_ArgMin_A} are presented. One can appreciate the smallness of the coefficients $\bar{a}$ due to the Tikhonov regularization and the rapid decrease in the coefficients $\bar{b}$ as the consequence of the applied orthonormalization.

\begin{figure}[H]
  \centering
    \begin{tabular}{cc}
      \resizebox{3in}{!}{\includegraphics{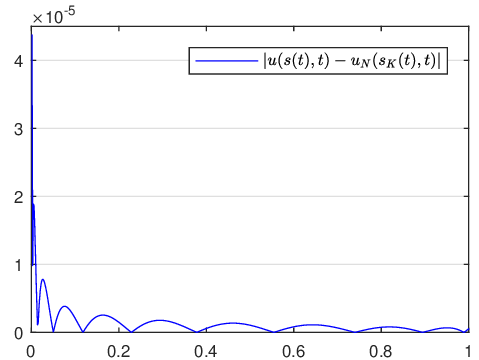}} &
      \resizebox{3in}{!}{\includegraphics{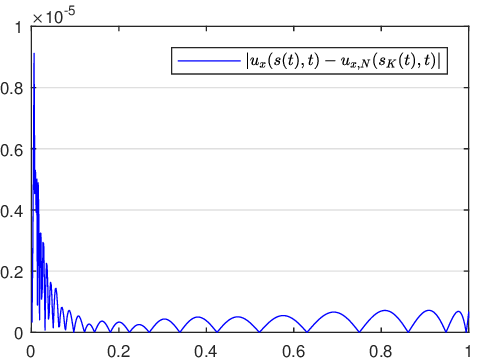}} \\
    \end{tabular}
  \caption{The typical approximation errors for the boundary conditions \eqref{eq_FBP_a1} and \eqref{eq_FBP_a2}, with the same parameters as used to produce Figure \ref{Fig_Ex_FB_OptionValue}.}
    \label{Fig_Errors}
\end{figure}


\begin{figure}[H]
  \centering
    \begin{tabular}{cc}
      \resizebox{3in}{!}{\includegraphics{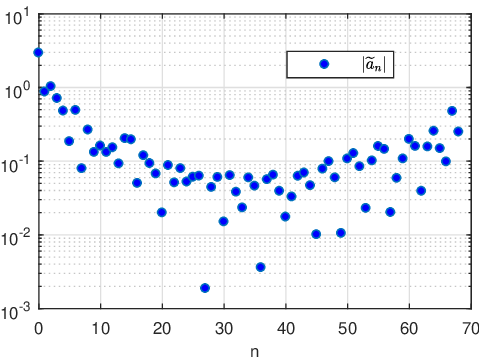}} &
      \resizebox{3in}{!}{\includegraphics{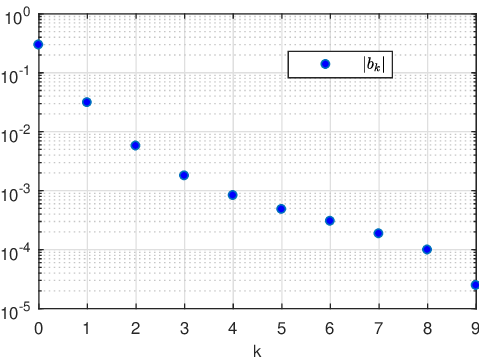}} \\
    \end{tabular}
  \caption{The typical absolute values of the coefficients $\widetilde{a}$ and $\bar{b}$ for Problem \ref{Problem_ArgMin_A}, with the same parameters as used to produce Figure \ref{Fig_Ex_FB_OptionValue}.}
    \label{Fig_Coefficients}
\end{figure}

\begin{table}[h]
\centering
{\small\setlength{\tabcolsep}{3pt} 
\begin{tabular}{l|llll|llll|llll} 
\hline
&\multicolumn{ 4}{c}{$u(0,T)$} &\multicolumn{ 4}{c}{$u(0.1,T)$} &\multicolumn{ 4}{c}{$u(0.2,T)$}\\
\hline
$T$ & TES & LCM & BTM & RIM & TES & LCM & BTM & RIM & TES & LCM & BTM & RIM\\
\hline
1/3 & 1.1340 &  & 1.1324 & 1.1335 & 1.0462 &  & 1.0452 & 1.0454 & 1.0065 &  & 1.0062 & 1.0064 \\
7/12 & 1.1744 &  & 1.1727 & 1.1742 & 1.0771 &  & 1.0761 & 1.0765 & 1.0208 &  & 1.0203 & 1.0203 \\
1 & 1.2237 & 1.2188 &  & 1.2235 & 1.1175 & 1.1125 &  &  & 1.0453 & 1.0426 &  &  \\
2 & 1.3078 &  &  &  & 1.1891 &  &  &  & 1.0968 &  &  &  \\
5 & 1.4401 & 1.4228 &  &  & 1.3049 & 1.2890 &  &  & 1.1892 & 1.1741 &  &  \\
10 & 1.5508 & 1.5273 &  &  & 1.4029 & 1.3816 &  &  & 1.2712 & 1.2517 &  &  \\
40 & 1.6831 &  &  &  & 1.5208 &  &  &  & 1.3718 &  &  &  \\
100 & 1.6904 &  &  &  & 1.5273 &  &  &  & 1.3775 &  &  &  \\
\hline
$\infty$&\multicolumn{ 4}{c}{$1.6904$} &\multicolumn{ 4}{c}{$1.5273$} &\multicolumn{ 4}{c}{$1.3769$}\\
\hline
\end{tabular}
}
\caption{Option value for Problem \ref{Problem_KimuraFBPv2}. The fixed parameters are $r=0.05$, $\delta=0.03$ and $\sigma_{0}=0.3$.}
\label{Table1_1}
\end{table}

In Table \ref{Table1_1} the values of the option for the different time horizons $T$ are shown, borrowing the parameter configuration of \citet[Table 1]{kimura2008valuing} and \citet[Table 1]{jeon2016integral}. One can appreciate an excellent agreement of the results produced by the proposed method with those delivered by the RIM and slightly worse agreement with the results produced by the BTM. The latter is due to the fact that even 10000 steps used is insufficient for the BTM to be precise to 4 figures. As for the  results from \cite{kimura2008valuing}, there are two concerns. First, the method used in \cite{kimura2008valuing} is based on the Laplace-Carlson transform and  requires the option value to be defined for any $t\in(0,\infty)$ and to satisfy an equation similar to \eqref{eq_DiffOp_Complete} for any $t>0$. That is, a solution should have a continuation across the free boundary satisfying the same initial condition at $t=0$. It is not clear why this rather strong assumption holds, and if not, how close is the obtained solution to the exact one. Second, the inversion of the Laplace-Carlson transform was computed by  the Gaver-Stehfest method which is rather delicate to implement and can result in relative errors as high as several percent, see \cite{kuznetsov2013} and references therein, no error analysis was presented. Nevertheless, our results are quite close to those of \cite{kimura2008valuing}.

We can also observe from Table \ref{Table1_1} that as $T$ increases the algorithm converges to the infinite horizon value. For $T=100$, we are already very close to the theoretical value of the perpetual option.

In Figure \ref{Fig_DiffTax_rAndSigma} the value of the option under different initial conditions is revealed. By the definition of $y$ in \eqref{eq_Var_y_t} the option is more valuable if the initial supremum of the process is the same as the initial value of the underlying, i.e. $s/m=1$. We present this curve under different financial parameters $\sigma $ and $r$, that can be compared with \citet[Figures 2 and 3]{jeon2016integral}.

\begin{figure}[h]
  \centering
    \begin{tabular}{cc}
      \resizebox{3in}{!}{\includegraphics{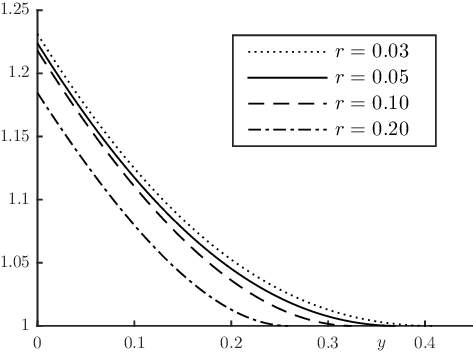}} &
      \resizebox{3in}{!}{\includegraphics{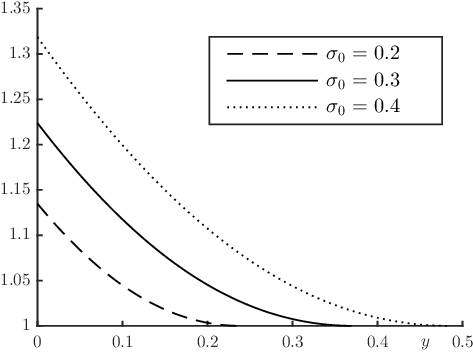}} \\
    \end{tabular}
  \caption{\small{Option value under different initial conditions. The common parameters are $T=1$, $\delta =0.03$. Left: $\sigma_{0}=0.3$. Right: $r=0.05$.}}
    \label{Fig_DiffTax_rAndSigma}
\end{figure}
\section{Final comments and future research}\label{sec_Conslusion}

In summary, the proposed method has a lot of potential for further financial engineering applications possessing path-dependency and early exercise
features such as lookback options, American options, etc. The method is not restricted to the BSM operator and can easily be applied to any other time-independent differential operator \eqref{eq_parabolIntro}.

\section*{Acknowledgements}

Research was supported by CONACYT, Mexico via the projects 222478 and 284470.
This work was supported by Funda\c{c}\~{a}o para a Ci\^{e}ncia e a Tecnologia, grant UIDB/00315/2020.
The first named author would like to express his gratitude to the Excellence
scholarship granted by the Mexican Government via the Ministry of Foreign
Affairs which gave him the opportunity to develop this work during his stay
in the CINVESTAV, Mexico. Research of Vladislav Kravchenko was partially supported by the
Regional mathematical center of the Southern Federal University with the Agreement 075--02--2022--893 of the Ministry of Science and Higher Education of Russia.
The authors thank the helpful comments and discussions of the participants at the 10th World Congress of the Bachelier Finance Society (Dublin, Ireland).

\section*{Conflicts of Interest}
The authors declare that there are no conflicts of interest regarding the publication of this paper.

\section*{Data availability}
The data that support the findings of this study are available upon reasonable request.


\bibliographystyle{model2-names}

\appendix

\renewcommand{\theequation}{\Alph{section}.\arabic{equation}}
\newtheorem{cor}{Corollary}[section]
\setcounter{equation}{0}
\setcounter{equation}{0}

\section{Appendix}

\subsection{\label{Sec_Ap_InfHor}Russian option with infinite horizon under the BSM model}

For the sake of completeness, we include the formula of \cite{shepp1993russian}
for the pricing of the perpetual Russian option. For $\delta >0$, the upper
boundary value is given by%
\begin{equation*}
b_{\infty }=1-\left( \frac{d_{2}(1-d_{1})}{d_{1}(1-d_{2})}\right) ^{\frac{1}{%
d_{1}-d_{2}}},
\end{equation*}%
where $d_{i}$, with $i\in{1,2}$, are the solutions to the quadratic equation
\begin{equation*}
\frac{1}{2}\sigma ^{2}x^{2}+(r-\delta -\frac{1}{2}\sigma ^{2})x-r=0.
\end{equation*}%
The value of the option is obtained from
\begin{equation*}
u_{\infty }=\frac{1}{d_{2}-d_{1}}\left\{ d_{2}\left( \frac{s}{b_{\infty }}%
\right) ^{d_{1}}-d_{1}\left( \frac{s}{b_{\infty }}\right) ^{d_{2}}\right\}.
\end{equation*}%
The detailed analysis of this problem can be consulted in \citet[Section VII, $\S $26]{PeskirShiryaev06}, \cite{kimura2008valuing} and the references
therein.

\subsection{Transmuted heat polynomials}

The heat polynomials are defined for $n\in\mathbb{N}$ as---see, e.g., \cite%
{rosenbloom1959expansions} and \cite{widder1962},
\begin{equation*}
h_{n}(x,t) =\sum \limits_{k=0}^{\left[ n/2\right]
}c_{k}^{n}x^{n-2k}t^{k},
\end{equation*}
where $\left[ \cdot\right] $ denotes the entire part of the number and
\begin{equation*}
c_{k}^{n}=\frac{n!}{\left( n-2k\right) !k!}.
\end{equation*}
The first five heat polynomials are%
\begin{align*}
h_{0}\left( x,t\right) & =1,\text{ \ \ \ }h_{1}\left( x,t\right) =x,\text{ \
\ \ }h_{2}\left( x,t\right) =x^{2}+2t, \\
h_{3}\left( x,t\right) & =x^{3}+6xt,\text{ \ \ \ }h_{4}\left( x,t\right)
=x^{4}+12x^{2}t+12t^{2}.
\end{align*}

The set of heat polynomials $\left\{ h_{n}\right\} _{n\in \mathbb{N}\cup
\left\{ 0\right\} }$ represents CSS for the heat equation%
\begin{equation}
u_{xx}=u_{t}  \label{Eq_Heat_1}
\end{equation}%
on any domain $D\left( s\right) $ defined by \eqref{D(s)_1}---see \cite{ColtonWatzlawek1977}.

Similarly to \cite{KrKrTorba2017}, we will call the functions $H_{n}=\mathbf{%
T}\left[ h_{n}\right] $ the \textbf{transmuted heat polynomials\footnote{In \cite{KrKrTorba2017} it is analyzed the case with $p\equiv 1$ and $r\equiv 1$
.}}. As corollary of Theorem \ref{Teor_Phi_x} we can show that $H_{n}$
are solutions to equation \eqref{eq_FBP1_1}, i.e., $\left( \mathbf{C}%
-\partial _{t}\right) H_{n}\left( y,t\right) =0$. Moreover, the set $\left\{
H_{n}\right\} _{n\in \mathbb{N}}$ is a CSS for \eqref{eq_FBP1_1} on any domain $%
D\left( s\right) $ defined by \eqref{D(s)_1} due to Proposition \ref{Prop Complete system_1} and the completeness of the
system of heat polynomials \cite{ColtonWatzlawek1977}.

\begin{cor}
\label{Cor_Hn_1} The transmuted heat polynomials admit the following form%
\begin{equation}
H_{n}(y,t) =\sum \limits_{k=0}^{\left[ n/2\right]
}c_{k}^{n}\Phi_{n-2k}\left( y\right) t^{k}.  \label{eq_Hn}
\end{equation}
\end{cor}

\begin{proof}
This equality is an immediate corollary of Theorem \ref{Teor_Phi_x}.
Indeed, we have $H_{n}(x,t) =\mathbf{T}[h_{n}(x,t)] =\sum_{k=0}^{\left[ n/2\right] }c_{k}^{n}\mathbf{%
T}[x^{n-2k}] t^{k}=\sum_{k=0}^{\left[ n/2\right]
}c_{k}^{n}\Phi_{n-2k}\left( y\right) t^{k}$, where Theorem \ref{Teor_Phi_x}
is used.
\end{proof}

The explicit form \eqref{eq_Hn} of the functions $H_{n}$ allows the
construction of the approximate solution to Problem \ref{Problem_FBP_1} by the
THP. The presented here is the extension of the results from \cite{KrKrTorba2017}.

\end{document}